\documentclass{article}
\usepackage{amsmath, amssymb, amsthm, epsf, graphicx, amscd, amsfonts, amscd, multirow}

\usepackage{setspace}
\newtheorem{theorem}{Theorem}[section]
\newtheorem{lemma}[theorem]{Lemma}
\newtheorem{proposition}[theorem]{Proposition}
\newtheorem{corollary}[theorem]{Corollary}

\def\bb #1{ {\mathbb #1} }
\def\c #1{ {\mathcal #1} }
\def\f #1{ {\mathfrak #1} }
\def\b #1{ {\bf #1} }

\begin{document}
\title{Families of generalized Kloosterman sums}

\author{C. Douglas Haessig \and Steven Sperber}

\date{\today}
\maketitle

\begin{abstract}
We construct $p$-adic relative cohomology for a family of toric exponential sums which generalize the classical Kloosterman sums. Under natural hypotheses such as quasi-homogeneity and nondegeneracy, this cohomology is acyclic except in the top dimension. Our construction enables  sufficiently sharp estimates for the action of Frobenius on cohomology so that our earlier work may be applied to the $L$-functions coming from linear algebra operations on these families to deduce a number of basic properties.
\end{abstract}


\section{Introduction}\label{S: Intro}

Our aim in the following is to construct relative $p$-adic cohomology for a family of toric exponential sums generalizing the classic hyperkloosterman sums. In the hyperkloosterman case, the relevant family of regular functions on $\bb G_m^n$ is given by
\[
K_n(\lambda, x) := x_1 + \cdots + x_n + \frac{\Lambda}{x_1 \cdots x_n}.
\]
Let $\Theta$ be a non-trivial additive character of the field $ {\bb F}_q $ of $q = p^a$ elements, $p$ a prime number. 
For any $\bar \lambda \in \overline{\bb F}_q^*$, the hyperkloosterman sum (over $\bb F_q(\bar \lambda)$) is given by
\[
Kl_{n+1}(\bar \lambda) := \sum_{\bar x \in (\bb F_q(\bar \lambda)^*)^n} \Theta_{\bar \lambda}( K_n(\bar \lambda, \bar x))
\]
where $\Theta_{\bar \lambda}$ (= $ \Theta \circ Tr_{\bb F_q(\bar \lambda)/{\bb F}_p} $) is a non-trivial additive character of $\bb F_q(\bar \lambda)$. In our generalization, we replace the linear form $\sum_{i=1}^n x_i$ by an arbitrary quasi-homogeneous nondegenerate Laurent polynomial $\bar f(x)$, and we replace the deforming monomial $\Lambda (x_1 \cdots x_n)^{-1}$ by an arbitrary monomial $ \Lambda x^\mu$ where $\mu \in \bb Z^n$ is a lattice point which does not lie on the affine hyperplane spanned by the monomials in the support of $\bar f$. 

Let $\bar f(x) = \sum \bar A(v) x^v \in \bb F_q[x_1^\pm, \ldots, x_n^\pm]$, and denote by $Supp(\bar f)$, the support of $\bar f$, the set of $v \in \bb Z^n$ such that $\bar A(v) \not= 0$. Let $\sigma = \Delta(\bar f)$ be the convex polytope which is the convex closure of $Supp(\bar f)$; let $\Delta_\infty(\bar f)$ be the convex closure of $\Delta(\bar f) \cup \{ 0 \}$. Let $Cone(\bar f) := Span_{\bb R_{\geq 0}} \Delta_\infty(\bar f)$, the cone over $\Delta_\infty(\bar f)$ consisting of the union of rays from 0 passing through $\Delta_\infty(\bar f)$. We assume throughout this work that $dim \> \Delta_\infty(\bar f) = dim \> Cone(\bar f) = n$. Set $M(\bar f) := Cone(\bar f) \cap \bb Z^n$.

Let $l_\sigma(x) = 1$ be the equation of the affine hyperplane spanned by $Supp(\bar f)$, with $l_\sigma(x) \in \bb Q[x_1, \ldots, x_n]$ a rational linear form. With respect to the Newton polytope $\Delta_\infty(\bar f)$ and the hyperplane $l_\sigma(x) = 1$, there are four possible places the deforming monomial $x^\mu$ can appear: above the hyperplane $l_\sigma(x) = 1$, on the hyperplane, or under the hyperplane but inside or outside of the Newton polytope; see Figure \ref{F: deform}.
\begin{figure}
\caption{Four possibilities for $\mu$}\label{F: deform}
  \centering
\includegraphics[width=5in]{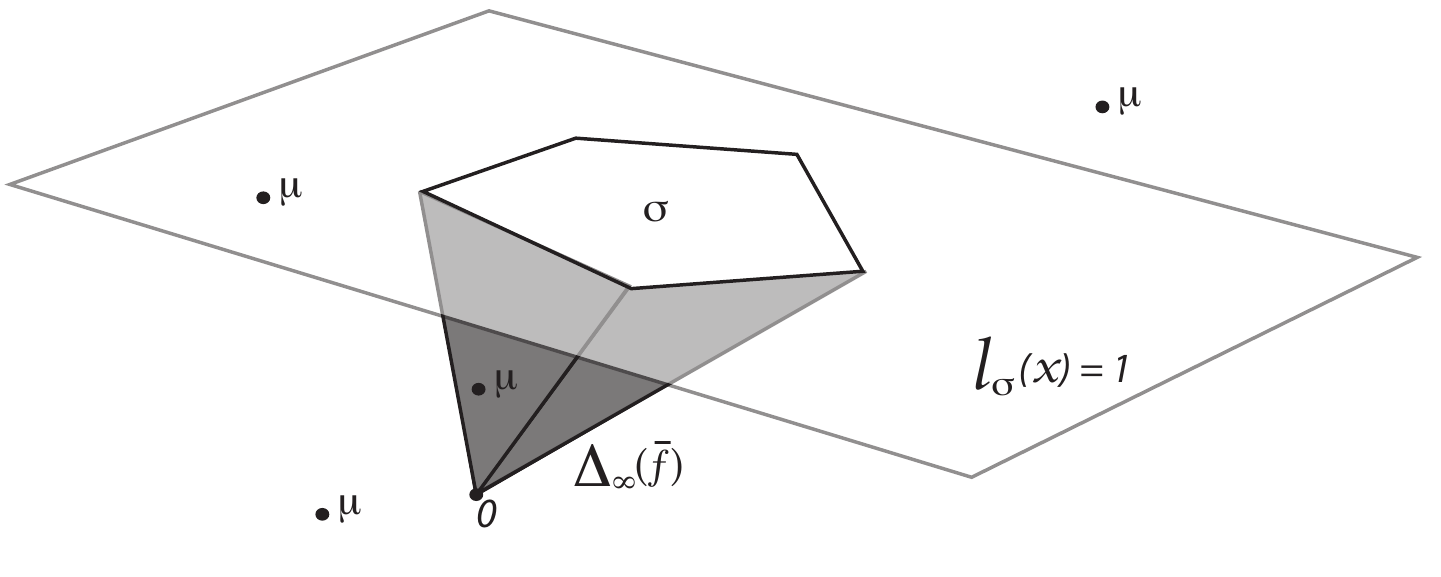}
\end{figure}
In the case when $\mu$ lies inside the Newton polytope, in particular $\mu \in M(\bar f)$ and $l_\sigma(\mu) < 1$, the family $\bar F(\Lambda, x) := \bar f(x) + \Lambda x^\mu$
is a lower order deformation and has been studied in \cite{H-S} in some detail. We will not consider this case in the present work. If $\mu$ lies on the hyperplane $l_\sigma(\mu) = 1$, it seems more difficult to obtain precise $p$-adic estimates and we do not treat this case here. Instead we focus here on the case $l_\sigma(\mu) > 1$, or $\mu \not \in M(\bar f)$ and $l_\sigma(\mu) < 1$.  Because the monomial $\mu$ is in our present case a vertex of $\Delta_\infty(\bar f, \mu) := $convex closure of $\Delta_\infty(\bar f) \cup \{\mu\}$, and will have polyhedral weight equal to 1 (like the other vertices, other than zero, of $\Delta_\infty(\bar f, \mu)$), we view the family
\[
\bar F(\Lambda, x) = \bar f(x) + \Lambda x^\mu
\]
as an ``isobaric'' deformation of the sum defined by $\bar f(x)$. 

Our goal in this work is to establish $p$-adic relative cohomology for our family with enough understanding of the relative Frobenius map on cohomology that we may use our earlier general results on $\sigma$-modules \cite{H-S} with polyhedral $p$-adic growth to establish some applications concerning degree (as rational function) and  total degree for the symmetric power $L$-functions (and $L$-functions associated with other linear algebra operations) for the given family. See Theorems \ref{T: Main} and \ref{T: Main2} in Section \ref{S: L-function} for details of these results.

In this work and in our earlier work \cite{H-S} on lower-order deformations of nondegenerate toric sums, we proceed cohomologically to construct $p$-adic relative cohomology (which vanishes in all but the top dimension). We then apply fairly arbitrary linear algebra operations to the remaining relative cohomology space and its Frobenius map, and construct a new complex satisfying a trace formula for the L-function associated with the linear algebra operation. Using methods which are seemingly special to Dwork's precohomological constructions here,  we obtain some  general results pertaining to degree and total degree of these L-functins by careful examination of the complexes on relative cohomology. We hope in a future work to compute the cohomology of this complex in the case of the classical hyperkloosterman family. In that setting then, we expect cohomological methods to give  more precise information. The hyperkloosterman family and its symmetric power $L$-functions have been studied previously by Fu-Wan \cite{FuWan-$L$-functionssymmetricproducts-2005}, \cite{FuWan-Functional}, \cite{FuWan-Trivial}, \cite{FuWan-KloostermanZ} using $\ell$-adic cohomology, and by Robba \cite{Robba-SymmetricPowersof-1986} in the case of the one dimensional Kloosterman family $K_1(\lambda, x) = x + \Lambda/x$ using $p$-adic cohomology. Also Tsuzuki in an unpublished work \cite{Tsuzuki-Kloosterman} has given another construction of relative cohomology for the Bessel functions using rigid cohomology. In addition, we believe it is possible to further relax the hypothesis that $\bar f(x)$ is quasihomogeneous and allow more complicated multi-parameter families.  

We thank Carmala Garzione for help with generating the figures.

\section{Fibered family over $\bb F_q$}\label{S: 1}

We keep the notations from the introduction above. If $\tau$ is a closed face of $\Delta_\infty(\bar f, \mu)$ not containing 0 we say $\tau $ is a \emph{face at $\infty$}. Assuming $\bar f$ is quasihomogeneous, then $\sigma = \Delta(\bar f)$ is the unique face of $\Delta_\infty(\bar f)$ of codimension one at $\infty$. Let $\mu \in \bb Z^n$ with $l_{\sigma}(\mu) \not= 1$. The cases $l_{\sigma}(\mu) < 1$ and $l_{\sigma}(\mu) > 1$ have some differences in our treatment. See Figure \ref{F: Fig1}.
\begin{figure}
\caption{Two cases}\label{F: Fig1}
  \centering
\includegraphics[width=5in]{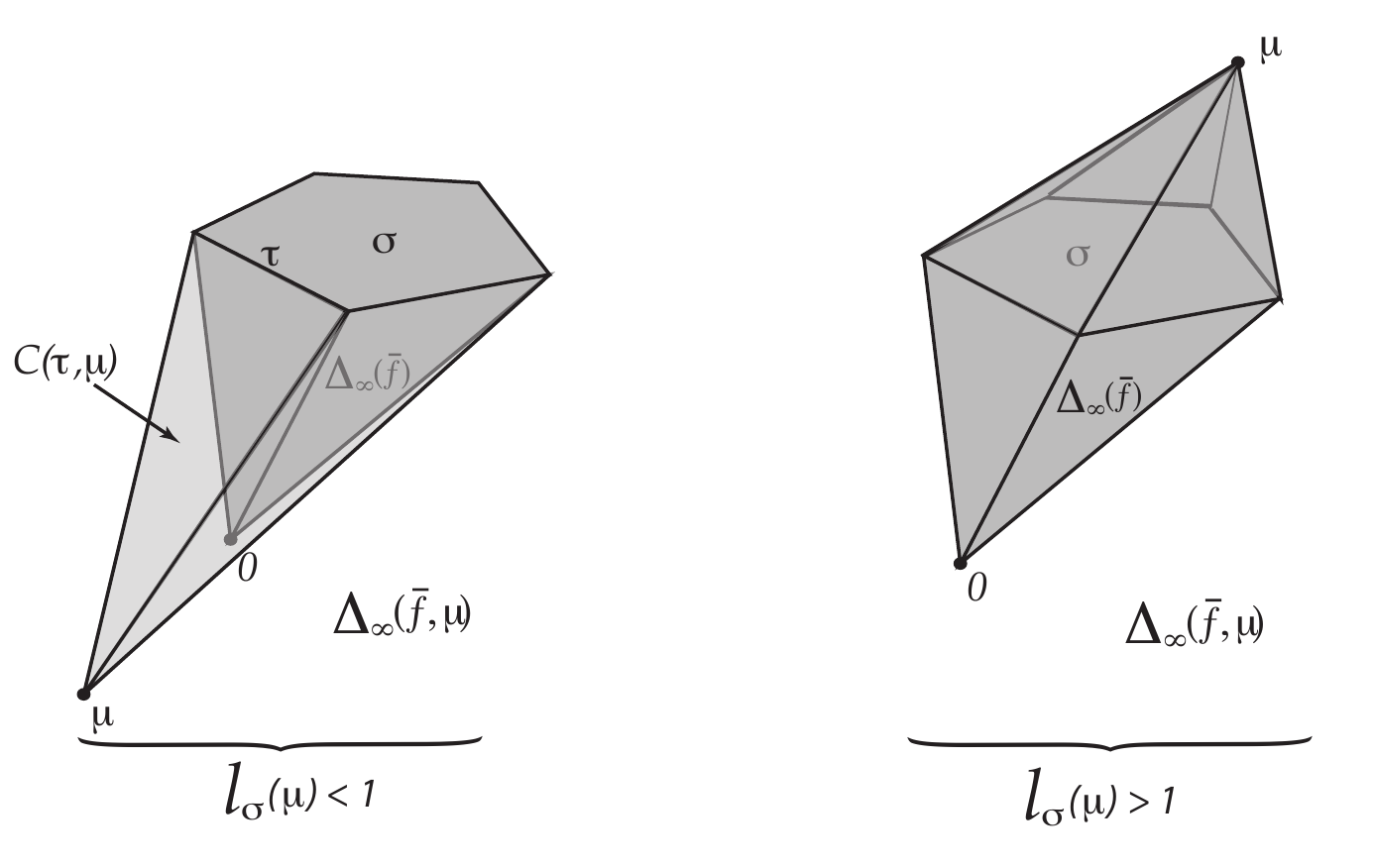}
\end{figure}
We think it is preferable to treat the case $l_{\sigma}(\mu) < 1$, $\mu \not\in Cone(\bar f)$ first and deal with the necessary changes to treat $l_{\sigma}(\mu) > 1$ afterwards. So we will from now on, unless stated explicitly otherwise, assume that $l_{\sigma}(\mu) < 1$.

It follows that the set of faces at $\infty$ of $\Delta_\infty(\bar f, \mu)$ contains as a subset the faces at $\infty$ of $\Delta_\infty(\bar f)$. The other faces of $\Delta_\infty(\bar f, \mu)$ at $\infty$ contain the vertex $\mu$. These faces are either ${\mu}$ itself or are of the form $\omega = C(\tau,\mu)$ (= the convex closure of $\mu$ and a face $\tau$ at  $\infty$ of $\Delta_\infty(\bar f)$). See Figure \ref{F: Fig1}. If $\omega$ has codimension $r$ in $\Delta_\infty(\bar f, \mu)$ then $\tau$ has codimension $r+1$ in $\Delta_\infty(\bar f)$. We say a face $\tau$ at $\infty$ of $\Delta_\infty(\bar f)$ is \emph{visible from $\mu$} if $C(\tau, \mu)$ is a face at $\infty$ of $\Delta_\infty(\bar f, \mu)$. We denote by $\Gamma$ the set of all codimension two faces $\tau$ at  $\infty$ of $\Delta_\infty(\bar f)$ and by $\Gamma_1$ the subset of codimension two faces $\tau$ at $\infty$ of $\Delta_\infty(\bar f)$ visible from $\mu$. For example, in the hyperkloosterman case, all faces of $\Delta_\infty(\bar f)$ (where  $\bar f = \sum_{i=1}^n x_i$), other than $\Delta(\bar f) $ itself, are visible from $\mu =(-1, \ldots, -1)$ so that $\Gamma_1 = \Gamma$ in this case.  On the other hand, with this same  $\bar f$, if $\mu = (-1,0, \ldots, 0)$ then $\Gamma_1$ consists of the unique codimension two face at $\infty$ of $\Delta_\infty(\bar f)$ spanned by {$ e_2, \ldots, e_n $}  where the $\{e_i\}$ denote the standard basis of $\bold {R}^n$. In particular, the faces at $\infty$ of $\Delta_\infty(\bar f, \mu)$ of codimension 1 are either $\sigma = \Delta(\bar f)$ itself or have the form $C(\tau, \mu)$ for  $\tau \in \Gamma_1$.

Every face $\tau$ of $\Delta_\infty(\bar f)$ of codimension 2 at $\infty$ spans an affine subspace of $\bb R^n$ of codimension 2 which may be described by the simultaneous linear equations $l_\sigma(x) = 1$ and $\phi^{(\tau)}(x) = 0$. Here, $\phi^{(\tau)}(x) = \sum_{i=1}^n b_i^{(\tau)} x_i \in \bb Z[x_1, \ldots, x_n]$ and we may assume $( b_1^{(\tau)}, \ldots, b_n^{(\tau)})$ are relatively prime. The linear subspace in $\bold{R}^n$ with equation $\phi^{(\tau)}(x) = 0$ is spanned by a boundary face of $\Delta_\infty(\bar f)$. The cone over $\tau$,  $C_\infty(\tau) = $ convex closure of $\tau \cup \{0\}$, in fact spans this boundary face.  We may and do assume the signs have been adjusted so that $Cone(\bar f)$ itself is given by the simultaneous inequalities $\phi^{(\tau)}(x) \geq 0$ where $\tau$ ranges over the codimension 2 faces at $\infty$ of $\Delta_\infty(\bar f)$. 

Recall if $\tau$ is any face of $\Delta_\infty(\bar f)$ we denote
\[
\bar f^{(\tau)}(x) := \sum_{v \in Supp(\bar f) \cap \tau} \bar A(v) x^v
\]
and we say $\bar f$ is nondegenerate along $\tau$ if the Laurent polynomials $\{ x_1 \frac{\partial \bar f^{(\tau)}}{\partial x_1}, \ldots, x_n \frac{\partial \bar f^{(\tau)}}{\partial x_n}\}$ have no common zero in $(\overline{\bb F}_q^*)^n$ where $\overline{\bb F}_q$ is an algebraic closure of $\bb F_q$. We say $\bar f$ is nondegenerate with respect to $\Delta_\infty(\bar f)$ if $\bar f$ is nondegenerate with respect to every closed face $\tau$  at $\infty$ of $\Delta_\infty(\bar f)$.

We now assume hypotheses H(i) through H(v) (note: the first four have already been mentioned):
\begin{itemize}
\item[] H(i). $dim \> \Delta_\infty(\bar f) = n$.
\item[] H(ii). $\bar f$ is quasihomogeneous.
\item[] H(iii). $\bar f$ is nondegenerate with respect to $\Delta_\infty(\bar f)$.
\item[] H(iv). $l_{\sigma}(\mu) < 1$. (See Section \ref{S: 3} for the case when $l_{\sigma}(\mu) > 1$.)
\item[] H(v).  $p  \nmid \prod_{\tau \in \Gamma_1} \phi^{(\tau)}(\mu)$, where $\Gamma_1$ is the set of codimension 2 faces at $\infty$ of $\Delta_\infty(\bar f)$ that are visible to $\mu$.
\end{itemize}
The last hypothesis is needed to prove the following.

\begin{theorem}\label{T: nondeg}
Let $\bar \lambda \in \overline{\bb F}_q^*$. Under the hypotheses H(i) through H(v), $\bar F(\bar \lambda, x) = \bar f(x) + \bar \lambda x^\mu$ is nondegenerate with respect to $\Delta_\infty(\bar F(\bar \lambda, x)) (= \Delta_\infty(\bar f, \mu))$. 
\end{theorem}

\begin{proof}
If $\omega$ is a face at $\infty$ of $\Delta_\infty(\bar F(\bar \lambda, x))$ not containing $\mu$ then $\omega$ is in fact a face of $\Delta_\infty(\bar f)$. Then
\[
\bar F^{(\omega)}(\bar \lambda, x) = \bar f^{(\omega)}(x)
\]
and $\bar F$ is nondegenerate along $\omega$ since $\bar f$ is, by hypothesis H(iii). If on the other hand $\mu$ belongs to the face $\omega$, then $\omega$ has the form $C(\gamma, \mu)$ with $\gamma$ a face at $\infty$ of $\Delta_\infty(\bar f)$. In fact $\gamma$ is contained in a face $\tau \in \Gamma_1$  and $\omega = C(\gamma, \mu) \subset C(\tau, \mu)$.  So in this case
\[
\bar F^{(\omega)}(\bar \lambda, x) = \bar f^{(\gamma)}(x) + \bar \lambda x^\mu.
\]
Let $\bar x \in \overline{\bb F}_q^n$ be a simultaneous zero of $\{ x_1 \frac{\partial \bar F^{(\omega)}}{\partial x_1}, \ldots, x_n \frac{\partial \bar F^{(\omega)}}{\partial x_n}\}$. Then $\bar x$ also satisfies
\begin{equation}\label{E: 2}
\left( \sum_{i=1}^n b_i^{(\tau)} x_i \frac{\partial}{\partial x_i} \right) (\bar F^{(\omega)}).
\end{equation}
But since all $v \in Supp(\tau)$ satisfy $\sum_{i=1}^n b_i^{(\tau)} v_i = 0$ it follows that $\sum_{i=1}^n b_i^{(\tau)} x_i \frac{\partial \bar f^{(\gamma)}(x)}{\partial x_i} = 0$  identically (where $\phi^{(\tau)}(x) = \sum_{i=1}^n b_i^{(\tau)} x_i$). As a consequence of this observation and (\ref{E: 2}), we get $\phi^{(\tau)}(\mu) \bar \lambda \bar x^u = 0$ in $\overline{\bb F}_q$. However by hypothesis H(v), $\phi^{(\tau)}(\mu) \not= 0 $ in $\bb F_q$ so that $\bar x_i = 0$ for some $i = 1, \ldots, n$. Hence $\bar F$ is nondegenerate.
\end{proof}

Let $Cone(\bar f)$ and $Cone(\bar f, \mu)$ be the cones in $\bb R^n$ over $\Delta_\infty(\bar f)$ and $\Delta_\infty(\bar f, \mu)$, respectively. Let $M(\bar f) = Cone(\bar f) \cap \bb Z^n$ and $M(\bar f, \mu) = Cone(\bar f, \mu) \cap \bb Z^n$ be the monoids of lattice points in the respective cones. Let $\bb F_q[M(\bar f, \mu)]$ be the monoid algebra of Laurent polynomials with coefficients in $\bb F_q$ and support in $M(\bar f, \mu)$. For $\bar \lambda \in \overline{\bb F}_q^*$, we sometimes will use the notation $\bb F_q^{(\bar \lambda)}$ to denote the field of $\bar \lambda$ over $\bb F_q$; that is, ${\bb F}_q^{(\bar \lambda)}:=\bb F_q(\bar \lambda)$ the field generated over $\bb F_q$ by $\bar \lambda$. Let $d(\bar \lambda) := [\bb F_q^{(\bar \lambda)} : \bb F_q]$ be the degree of $\bar \lambda$ over $\bb F_q$. We use the polyhedron $\Delta_\infty(\bar f, \mu)$ to define (in the usual way) a weight function on $M(\bar f, \mu)$: for $v \in M(\bar f, \mu)$ let $w(v)$ denote the smallest non-negative rational number $b$ such that $v \in b \cdot \Delta_\infty(\bar f, \mu)$. Note that there is a positive integer $d$ such that $w(M(\bar f, \mu)) \subset (1/d)\bb Z_{\geq 0}$. Using the weight function $w$, $R^{(\bar \lambda)} := \bb F_q^{(\bar \lambda)}[M(\bar f, \mu)]$ has an increasing filtration indexed by $(1/d)\bb Z_{\geq 0}$. More precisely, for $i \in (1/d)\bb Z_{\geq 0}$ set
\[
Fil_i(R^{(\bar \lambda)}) := \left\{ \sum_{v \in M(\bar f, \mu)} \bar A(v) x^v \in R^{(\bar \lambda)} \mid \bar A(v) \in \bb F_q^{(\bar \lambda)}, w(v) \leq i \right\}.
\]
Let $\bar R^{(\bar \lambda)}$ denote the associated graded ring. We construct two complexes as follows. The spaces in both cases are the same:
\[
\Omega^i(\bar R^{(\bar \lambda)}, \nabla(\bar F(\bar \lambda, x))) := \Omega^i(\bar R^{(\bar \lambda)}, \nabla(D^{(\bar \lambda)}))
:= \bigoplus_{1 \leq j_1 < j_2 < \cdots  < j_i \leq n} \bar R^{(\bar \lambda)} \frac{d x_{j_1}}{x_{j_1}} \wedge \cdots \wedge \frac{d x_{j_i}}{x_{j_i}}
\]
with the respective boundary operators given by
\[
\nabla(\bar F(\bar \lambda, x))(\xi \frac{d x_{j_1}}{x_{j_1}} \wedge \cdots \wedge \frac{d x_{j_i}}{x_{j_i}}) := \left( \sum_{l=1}^n x_l \frac{\partial \bar F(\bar \lambda, x)}{\partial x_l}\xi \frac{d x_l}{x_l} \right) \wedge \frac{d x_{j_1}}{x_{j_1}} \wedge \cdots \wedge \frac{d x_{j_i}}{x_{j_i}}
\]
in the one case, and
\[
\nabla(D^{(\bar \lambda)})(\xi \frac{d x_{j_1}}{x_{j_1}} \wedge \cdots \wedge \frac{d x_{j_i}}{x_{j_i}}) := \left( \sum_{l=1}^n (D_l^{(\bar \lambda)} \xi )\frac{ dx_l}{x_l} \right)\wedge \frac{d x_{j_1}}{x_{j_1}} \wedge \cdots \wedge \frac{d x_{j_i}}{x_{j_i}}
\]
where
\[
D_l^{(\bar \lambda)} := x_l \frac{\partial}{\partial x_l} + x_l \frac{\partial \bar F(\bar \lambda, x)}{\partial x_l}
\]
in the other case.

It follows then from Theorem \ref{T: nondeg} and \cite{AdolpSperb-ExponentialSumsand-1989}:

\begin{theorem}\label{T: R decomp}
For every choice $\bar \lambda \in  \overline{\bb F}_q^* $, both complexes are acyclic except in top dimension $n$. In both cases, the top dimensional cohomology $H^n$ is a finite free $\bb F_q^{(\bar \lambda)}$-algebra of rank $n! Vol(\Delta_\infty(\bar f, \mu))$. For each $i \in (1/d)\bb Z_{\geq 0}$ we may choose a monomial basis $B_i^{(\bar \lambda)}$ consisting of monomials of weight i for an $\bb F_q^{(\bar \lambda)}$-vector space $V_i^{(\bar \lambda)}$ such that the $i$-th graded piece $\bar R_i^{(\bar \lambda)}$ of $\bar R^{(\bar \lambda)}$ may be written as
\[
\bar R_i^{(\bar \lambda)} = V_i^{(\bar \lambda)} \oplus \sum_{l=1}^n x_l \frac{\partial \bar F(\bar \lambda, x)}{\partial x_l} \bar R_{i-1}^{(\bar \lambda)}
\]
so that if $B^{(\bar \lambda)} = \bigcup_{i \in (1/d)\bb Z_{\geq 0}} B_i^{(\bar \lambda)}$ and $V^{(\bar \lambda)} = \sum_{i \in (1/d)\bb Z_{\geq 0}} V_i^{(\bar \lambda)}$ is the  $\bb F_q^{(\bar \lambda)}$-vector space with basis $B^{(\bar \lambda)}$, then 
\[
\bar R^{(\bar \lambda)} = V^{(\bar \lambda)} \oplus \sum_{l=1}^n x_l \frac{\partial \bar F(\bar \lambda, x)}{\partial x_l} \bar R^{(\bar \lambda)}.
\]
It follows as well that
\[
\bar R^{(\bar \lambda)} = V^{(\bar \lambda)} \oplus \sum_{l=1}^n D_l^{(\bar \lambda)} \bar R^{(\bar \lambda)}.
\]
\end{theorem}
(Note that $\bar F(\bar \lambda, x)$ is homogeneous of weight 1,  but  the $D_l^{(\bar \lambda)}$ are not homogeneous operators.) We will see in Theorem \ref{T: 3.3}  that $B_i^{(\bar \lambda)}$ and $B^{(\bar \lambda)}$ may be chosen independently of $\bar \lambda \in \overline{\bb F}_q^*$. 

It is useful to express explicitly the weight function $w$ on the various chambers or closed subcones of $Cone(\bar f, \mu)$ corresponding to the codimension one faces $\omega$ of $\Delta_\infty(\bar f, \mu)$ at $\infty$. When $\omega = \sigma = \Delta(\bar f)$ and $v \in Cone(\bar f) \cap \bb Z^n = M(\bar f)$, then $w(v) = l_\sigma(v)$. If $\omega = C(\tau, \mu)$ and $\tau$ is a face of codimension 2 at $\infty$ of $\Delta_\infty(\bar f)$, then the affine hyperplane spanned by $C(\tau, \mu)$ in $\bb R^n$ has equation $l_{(\tau, \mu)}(x) = 1$ where 
\begin{equation}\label{E: 3}
l_{(\tau, \mu)}(x) = \frac{\phi^{(\tau)}(x)}{\phi^{(\tau)}(\mu)}(1 - l_\sigma(\mu)) + l_\sigma(x),
\end{equation}
a rational linear form in $x_1, \ldots, x_n$. See Figure \ref{F: Fig2}.
\begin{figure}
\caption{The hyperplane $l_{(\tau, \mu)}(x) = 1$}\label{F: Fig2}
  \centering
\includegraphics[width=2in]{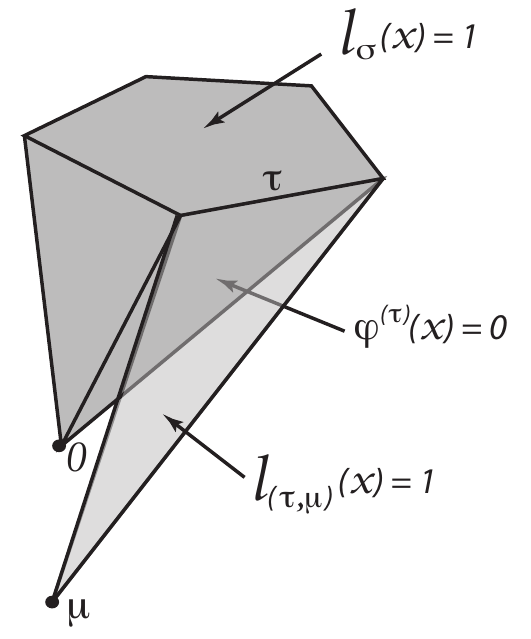}
\end{figure}
(Note, hypothesis H(v) precludes the possibility that $\mu$ belongs to any boundary face $\tau$ with equation $\phi^{(\tau)}(x) = 0$ of $Cone(\bar f)$ for which $\tau$ is visible from $\mu$.) \pagebreak If $Cone(\tau, \mu)$ is the closed cone of rays from the origin passing through $C(\tau, \mu)$, and $v \in Cone(\tau, \mu) \cap \bb Z^n$ ($=: M(\tau, \mu)$) then $w(v) = l_{(\tau, \mu)}(v)$.

\section{The extended monoid $\tilde M(\bar F)$}\label{S: monoid}

To compute relative cohomology we view $\bar F(\Lambda, x)$ as a Laurent polynomial in $n+1$ variables. To this end, we define $\Delta_\infty(\bar F) :=$ convex closure in $\bb R^{n+1}$ of $\{0\} \cup Supp(\bar F)$. It is convenient to fix the ordering of coordinates in $\bb R^{n+1}$ in the order $(\Lambda; x_1, \ldots, x_n)$ so that
\[
Supp(\bar F) = \{(1; \mu)\} \cup  \{ (0; v) \mid v \in Supp(\bar f))\} .
\]
Let $Cone(\bar F)$ be the cone in $\bb R^{n+1}$ over $\Delta_\infty(\bar F)$, and $M(\bar F) := Cone(\bar F) \cap \bb Z^{n+1}$, the monoid of lattice points in $Cone(\bar F)$. $\bar F$ is quasihomogeneous in $\bb R^{n+1}$ with all elements of $Supp(\bar F)$ lying on the affine hyperplane $W(\Lambda, x) = 1$, where
\begin{equation}\label{E: 4}
W(\Lambda, x) := l_\sigma(x) + \Lambda(1 - l_\sigma(\mu)).
\end{equation}
Note that $\Lambda$ itself, which is the point $(1; 0, \ldots, 0) \in M(\bar F)$, belongs to $Cone(\bar F)$ in certain cases (in fact, if and only if $l_\sigma(\mu) < 0$ and the line joining 0 and $\mu$ passes through the interior of $Cone(\bar f)$). Since we wish to proceed more generally, and since we will want to work over the polynomial ring in $\Lambda$ (or $\Lambda^{1/D}$ for a suitable positive integer $D$) we define a larger cone which contains the positive $\Lambda$-axis. The hyperplane $W(\Lambda, x) = 1$ meets the positive $\Lambda$-coordinate axis at the point $P_0 = ((1 - l_\sigma(\mu))^{-1}, 0, \ldots, 0)$. See Figure \ref{F: Fig5}.
\begin{figure}
\caption{The hyperplane $W(\Lambda, x)= 1$ and the point $P_0$}\label{F: Fig5}
  \centering
\includegraphics[width=5in]{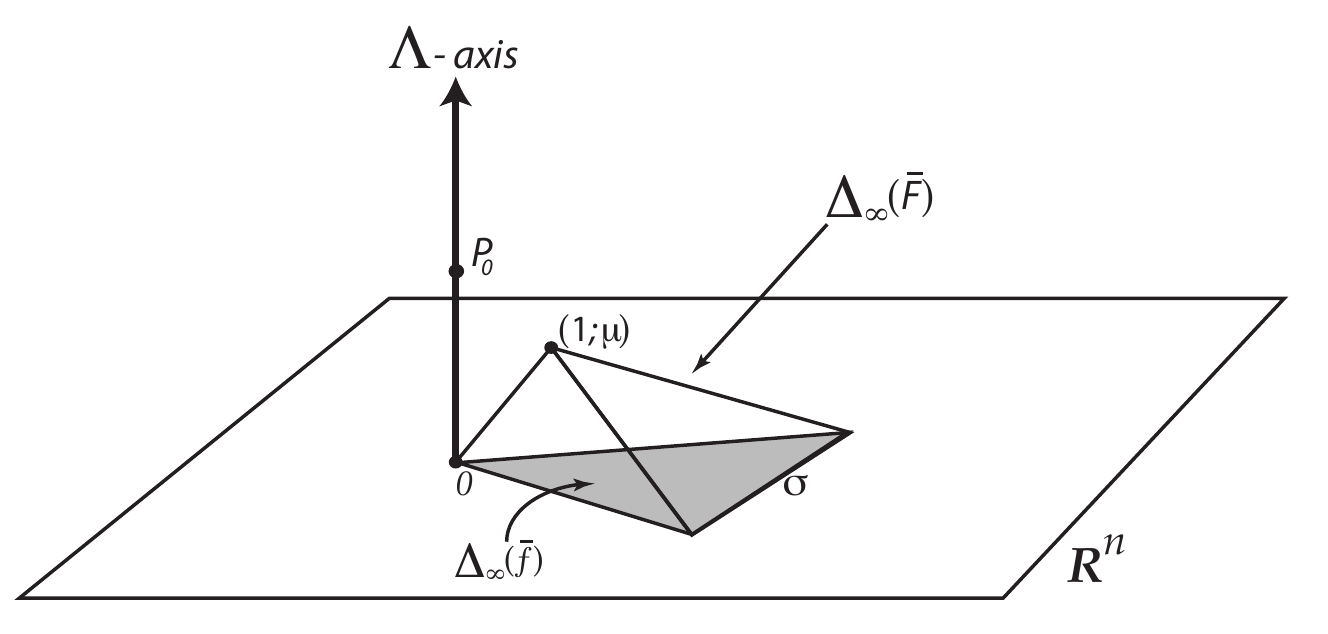}
\end{figure}
Set $\Delta_\infty(\bar F, P_0) := $ convex closure of $\Delta_\infty(\bar F) \cup \{P_0\}$ and $Cone(\bar F, P_0) := $ cone in $\bb R^{n+1}$ over $\Delta_\infty(\bar F, P_0)$. For later purposes, we define the extended monoid $\tilde M(\bar F) = Cone(\bar F, P_0) \cap \tilde L$, where $\tilde L$ is the extended lattice $((1/D) \bb Z) \times \bb Z^n$, where $D$ is the least common multiple of the set of positive integers $\{ \phi^{(\tau)}(\mu) \}_{\tau \in \Gamma_1}$,  (recall $\Gamma_1$ is the collection of codimension 2 faces of $\Delta_\infty(\bar f)$ at $\infty$ which are visible from $\mu$). $W(\Lambda, x)$ defines a weight function on $\tilde M(\bar F)$, and let $e$ be the least positive integer, divisible by $D$, such that $W(\tilde M(\bar F)) \subset(1/e)\bb Z_{\geq 0}$. 

The following definition plays an important role in our analysis.
\begin{equation}\label{E: 40}
m(v) := 
\begin{cases}
\phi^{(\tau)}(v)  / \phi^{(\tau)}(\mu) & \text{for } v \in Cone(\tau, \mu) \text{ and }  \tau \in \Gamma_1, \\
0 & \text{for }  v \in Cone(\bar f).
\end{cases}
\end{equation}
Note that $m(v) \geq 0$, since for $\tau \in \Gamma_1$ and $v \in Cone(\tau, \mu)$, $\phi^{(\tau)}(v) \leq 0$  and  $\phi^{(\tau)}(\mu)$ is negative. 

\begin{lemma}
$\tilde M(\bar F)$ may be described explicitly by:
\begin{equation}\label{E: extended monoid}
\tilde M(\bar F) = \{ (r, v) \in (1/D)\bb Z_{\geq 0} \times \bb Z^n | v \in Cone(\bar f, \mu), r \geq m(v) \}.
\end{equation}
\end{lemma}

\begin{proof}
Consider the projection $pr: \bb R^{n+1} \rightarrow \bb R^n$ onto the last $n$ coordinates. Then $pr( Cone(\bar F, P_0)) = Cone(\bar f, \mu)$. Let $\tau$ be any face of $\Delta_\infty(\bar f)$ at $\infty$ of codimension 2. Together with the origin, $\tau$ spans a hyperplane $H(\tau)$ with equation $\phi^{(\tau)}(x) = 0$ of dimension $n-1$ in $\bb R^n$. We view $\bb R^n$ and $H(\tau)$ in $\bb R^{n+1}$ via the embedding $\iota$ which sends $\bb R^n$ to the hyperplane $\Lambda = 0$ in $\bb R^{n+1}$ by the map
\[
\iota: (x_1, \ldots ,x_n) \mapsto  (0; x_1, \ldots, x_n)
\]
The hyperplane $H_{(\tau, \mu)}$ in $\bb R^{n+1}$ with equation $\Lambda \phi^{(\tau)}(\mu) - \phi^{(\tau)}(x) = 0$ is a hyperplane in $\bb R^{n+1}$, spanned by $\iota(H(\tau)) \cup \{ (1; \mu)\}$.

Claim:  $\tau$ is visible from $\mu$ if and only if $\phi^{(\tau)}(\mu) < 0$. To prove this, suppose first that $C(\tau, \mu)$ is a face of $\Delta_\infty(\bar f, \mu)$.  Then $\tau$ is a common face of $C(\tau, \mu)$ and $\sigma = \Delta(\bar f)$ so these codimension one faces at $\infty$ of $\Delta_\infty(\bar f, \mu)$ are on opposite sides of the hyperplane $H(\tau)$ with equation $\phi^{(\tau)}(x) = 0$. Since the signs are normalized so that elements of $Cone(\bar f)$ satisfy $\phi^{(\tau)}(x) \geq 0$ for all $\tau \in \Gamma$, it follows that $\phi^{(\tau)}(\mu) < 0$. Conversely, if $\phi^{(\tau)}(\mu) < 0$, consider 
\[
l_{(\tau, \mu)}(x) = l_\sigma(x) + (1 - l_\sigma(\mu)) \frac{\phi^{(\tau)}(x)}{\phi^{(\tau)}(\mu)}.
\]
For $x_0 \in \Delta_\infty(\bar f)$, $l_\sigma(x_0) \leq 1$ and $\phi^{(\tau)}(x_0) \geq 0$ so that $l_{(\tau, \mu)}(x_0) \leq 1$ (and the inequality is strict if $x_0 \not\in \tau$). So $C(\tau, \mu)$ is a face of codimension 1 at $\infty$ for $\Delta_\infty(\bar f, \mu)$.  This finishes the proof of the claim.

From this, it follows that the boundary faces for $Cone(\bar F)$ are given by $\Lambda = 0$, and $\Lambda \phi^{(\tau)}(\mu) - \phi^{(\tau)}(x) = 0$, and $Cone(\bar F)$ is described by the simultaneous inequalities 
\[
\begin{cases}
\Lambda \geq 0;  &  \\
\Lambda \geq \frac{\phi^{(\tau)}(x)}{\phi^{(\tau)}(\mu)} & \text{for $\tau$ visible from $\mu$, $(\tau \in \Gamma_1)$}; \\
\frac{\phi^{(\tau)}(x)}{\phi^{(\tau)}(\mu)} \geq \Lambda & \text{for $\tau \in \Gamma - \Gamma_1$}.
\end{cases}
\]

On the other hand the cone $Cone(\bar F, P_0)$ is described by the simultaneous inequalities
\[
\begin{cases}
\Lambda \geq 0; &  \\
\Lambda \geq \frac{\phi^{(\tau)}(x)}{\phi^{(\tau)}(\mu)} & \text{for $\tau$ visible from $\mu$, $(\tau \in \Gamma_1)$}; \\
\phi^{(\tau)}(x) \geq 0 & \text{for all } \tau \in \Gamma - \Gamma_1.
\end{cases}
\]
In particular,  the projection map $pr$ maps  $ Cone(\bar F, P_0)$ onto $Cone(\bar f, \mu)$ and
\[
pr^{-1}(Cone(\tau, \mu)) = \{ (\Lambda, x) \mid x \in Cone(\tau, \mu), \Lambda \geq \phi^{(\tau)}(x)  / \phi^{(\tau)}(\mu)\}
\]
which finishes the proof by (\ref{E: 40}).
\end{proof}

\section{Total space}\label{S: 2}

Using notation from the previous section, let $T := \bb F_q[\tilde M(\bar F)]$ be the graded $\bb F_q$-algebra with its grading given by $W$ and indexed by $(1/e) \bb Z_{\geq 0}$. In fact, it is an $(S := \bb F_q[\Lambda^{1/D}])$-algebra.  $S$ in fact is contained in T and is itself a graded $\bb F_q$-algebra via the restriction of $W$ to $S$. The gradings satisfy $S_i T_j \subset T_{i+j}$ where $i, j \in (1/e)\bb Z_{\geq 0}$. From (\ref{E: extended monoid}), it follows that $T = \bb F_q[\tilde M(\bar F)]$ is a free $S$-algebra with basis $\{ \Lambda^{m(v)} x^v\}_{v \in M(\bar f, \mu)}$. Furthermore, the gradings on $T$ (and $S$)  are  given by
\[
W(r, v) := l_\sigma(v) + r(1 - l_\sigma(\mu))
\]
for $\Lambda^r x^v \in {\tilde M(\bar F)}$. Finally, note that
\begin{align}
W(m(v), v) &= l_\sigma(v) + \frac{\phi^{(\tau)}(v)}{\phi^{(\tau)}(\mu)}(1 - l_\sigma(\mu)) \notag \\
&=
\begin{cases}\label{E: Wt}
l_{(\tau, \mu)}(v) = w(v) & \text{for $v \in Cone(\tau, \mu)$} \\
l_\sigma(v) = w(v) & \text{for $v \in Cone(\bar f)$, since $m(v) = 0$ in this case.}
\end{cases}
\end{align}
Let $\f m = (\Lambda^{1/D})$ be the maximal ideal in $S$. It is a homogeneous ideal in $S$ and $\f m T$ is the homogeneous ideal it generates in $T$. The following is basic in what follows.

\begin{theorem}\label{T: 3.1}
$\bar T := T/\f mT$ is an $\bar S := S / \f m (= \bb F_q)$ algebra with basis $\{ \Lambda^{m(v)} x^v \mid v \in M(\bar f, \mu) \}$. Further, $ \bar T$ is a graded $ \bb F_q$-algebra  with grading defined by
\[
\bar W( {[m(v); v]}) := W(m(v); v)
\]
where $[m(v);v]$ on the left denotes the class of $\Lambda^{m(v)} x^v$ in the quotient $T/\f mT$. For $\bar \lambda \in \overline{\bb F}_q^*$, the $\bb F_q^{(\bar \lambda)}$-linear map
\[
pr_{\bar \lambda}: T \otimes_{\bb F_q} \bb F_q^{(\bar \lambda)} \rightarrow \bar R^{(\bar \lambda)}
\]
via
\begin{align*}
pr_{\bar \lambda}(\Lambda^r x^v) &= (\Lambda^{r-m(v)} |_{\Lambda = 0}) \bar \lambda^{m(v)} x^v \\
&=
\begin{cases}
\bar \lambda^{m(v)} x^v & \text{if } r = m(v) \\
0 & \text{otherwise}
\end{cases}
\end{align*}
induces an isomorphism of graded $\bb F_q^{\bar \lambda}$-algebras between $\bar T \otimes_{\bb F_q} \bb F_q^{(\bar \lambda)}$ and $\bar R^{(\bar \lambda)}$.
\end{theorem}

\begin{proof}
This map is clearly onto, with kernel equal to $\f m T \otimes_{\bb F_q} \bb F_q^{(\bar \lambda)}$. So it defines on the quotient $\bar T \otimes_{\bb F_q} \bb F_q^{(\bar \lambda)}$ an isomorphism which respects the gradings. One may elaborate the multiplication in $\bar T$ by remarking that
\begin{equation}\label{E: 41}
w(v) = W(m(v); v) = l_\sigma(v) + m(v)(1 - l_\sigma(\mu))
\end{equation}
for $v \in M(\bar f, \mu)$. So writing
\[
(\Lambda^{m(v_1)} x^{v_1})(\Lambda^{m(v_2)} x^{v_2}) = \Lambda^{m(v_1) + m(v_2) - m(v_1 + v_2)} \Lambda^{m(v_1 + v_2)} x^{v_1 + v_2}
\]
we infer from the well-known properties of $w$ on $M(\bar f, \mu)$ that
\[
m(v_1) + m(v_2) - m(v_1 + v_2) \geq 0
\]
with equality holding if and only if $v_1$ and $v_2$ are cofacial in $M(\bar f, \mu)$. 
\end{proof}

We construct two complexes of $S$-algebras $\Omega^\bullet(T, \nabla(\bar F))$ and $\Omega^\bullet(T, \nabla(D))$ analogous to the complexes of $\bb F_q^{(\bar \lambda)}$-algebras defined in Section \ref{S: 1}. The spaces in both complexes are the same, namely
\[
\Omega^i := \bigoplus_{1 \leq j_1 < \cdots < j_i \leq n} T \frac{d x_{j_1}}{x_{j_1}} \wedge \cdots \wedge \frac{dx_{j_i}}{x_{j_i}},
\]
with boundary operators defined by
\[
\nabla(\bar F)(\xi  \frac{d x_{j_1}}{x_{j_1}} \wedge \cdots \wedge \frac{dx_{j_i}}{x_{j_i}}) := \left(\sum_{l=1}^n x_l \frac{\partial \bar F(\Lambda, x)}{\partial x_l} \xi \frac{dx_l}{x_l} \right)  \wedge  \frac{d x_{j_1}}{x_{j_1}} \wedge \cdots \wedge \frac{dx_{j_i}}{x_{j_i}}
\]
and
\[
\nabla(D)(\xi  \frac{d x_{j_1}}{x_{j_1}} \wedge \cdots \wedge \frac{dx_{j_i}}{x_{j_i}} ) :=  \left(\sum_{l=1}^n D_l(\Lambda, x) \xi  \frac{dx_l}{x_l} \right) \wedge  \frac{d x_{j_1}}{x_{j_1}} \wedge \cdots \wedge \frac{dx_{j_i}}{x_{j_i}},
\]
where
\[
D_l(\Lambda, x) = x_l \frac{\partial}{\partial x_l} + x_l \frac{\partial \bar F(\Lambda, x)}{\partial x_l}.
\]

\begin{theorem}\label{T: 3.2}
The complexes $\Omega^\bullet(T, \nabla(\bar F))$ and $\Omega^\bullet(T, \nabla(D))$ of $S$-algebras are acyclic except in dimension $n$. $H^n(\Omega^\bullet(T, \nabla(\bar F)))$ is a free, graded $S$-algebra of finite rank $n! Vol \> \Delta_\infty(\bar f, \mu)$. $H^n(\Omega^\bullet(T, \nabla(D)))$ is a free filtered $S$-algebra of finite rank $n! Vol \Delta_\infty(\bar f, \mu)$. If we fix $\bar \lambda \in \bb F_q^*$ (for simplicity) and set $\bar B := \{ \Lambda^{m(v)} x^v \mid v \in B^{(\bar \lambda)} \}$ and $\tilde V$ the free $S$-module with basis $\bar B$ then
\[
T = \tilde V \oplus \sum_{i=1}^n x_i \frac{\partial \bar F(\Lambda, x)}{\partial x_i} T
\]
and
\[
T = \tilde V \oplus \sum_{i=1}^n D_i T.
\]
\end{theorem}

\begin{proof}
We begin by showing $\{ x_i \frac{\partial \bar F}{\partial x_i} \}$ is a regular sequence in $T$. Suppose 
\begin{equation}\label{E: 5}
\sum_{i=1}^n x_i \frac{\partial \bar F(\Lambda, x)}{\partial x_i} \xi_i = 0
\end{equation}
with $\{ \xi_i \}_{i=1}^n \in T$. Since $x_i \frac{\partial \bar F}{\partial x_i}$ are homogeneous elements of weight 1, we may assume the $\{ \xi_i\}$ are themselves homogeneous of weight, say, $r$.

Any monomial $\Lambda^r x^v \in \bb F_q[\tilde M(\bar F)]$ may be factored $\Lambda^{r-m(v)}( \Lambda^{m(v)} x^v)$. We say $ord_\Lambda(\Lambda^r x^v) := r - m(v) \in (1/D)\bb Z_{\geq 0}$. We define the $\Lambda$-weight of the monomial $\Lambda^r$, $r \in (1/D)\bb Z_{\geq 0}$, by setting $W_\Lambda(\Lambda^r) := r(1 - l_\sigma(\mu))$. It follows then that
\begin{align*}
W(\Lambda^r x^v) &=  w(v) + W_\Lambda( \Lambda^{r-m(v)}) \\
&=  w(v) + (r-m(v))(1-l_\sigma(\mu)).
\end{align*}
We refer to $W(\Lambda^r x^v)$ as the total weight of the monomial $\Lambda^r x^v$.  Suppose $\xi$ is homogeneous of (total) weight $r \in (1/e) \bb Z_{\geq 0}$. We may write $\xi = \sum_{i = i_0}^k \Lambda^{i/D} \xi^{(i)}$, where $\xi^{(i)} \in T$ is homogeneous of (total) weight $r - (i/D)(1 - l_\sigma(\mu))$ and the terms in the support of $\xi^{(i)}$ all have $ord_\Lambda = 0$. The upper limit $k$ in the sum must be $\leq \lfloor \frac{rD}{1 - l_\sigma(\mu)} \rfloor$ and without loss of generality we may write $k = \lfloor \frac{rD}{1 - l_\sigma(\mu)} \rfloor$. 

Assume now in (\ref{E: 5}), $\min_{1 \leq j \leq n}\{ ord_\Lambda(\xi_j) \} = i_0 / D$ so we may write
\begin{equation}\label{E: 7}
\xi_j = \sum_{i=i_0}^k \Lambda^{i/D} \xi_j^{(i)},
\end{equation}
each $\xi_j^{(i)}$  homogeneous of total weight $r - (i/D)(1-l_\sigma(\mu))$ and $ord_\Lambda( \xi^{(i)}_j) \geq 0$.  We will prove there is a skew-symmetric set $\{ \eta_{jl} \} \subset T$ and elements $\{ \tilde \xi_j \}_{j=1}^n \subset T$ with $\tilde \xi_j$ homogeneous of total weight $r$, such that
\begin{equation}\label{E: 8}
\xi_j = \tilde \xi_j + \sum_{l=1}^n x_l \frac{\partial \bar F(\Lambda, x)}{\partial x_l} \eta_{jl}
\end{equation}
and 
\[
\min_{1\leq j\leq n} \left\{ ord_\Lambda(\tilde \xi_j) \right\} \geq (i_0 + 1)/D.
\]
Assume we have done this. Then since $\{ \eta_{jl}\}$ is skew-symmetric, $\sum_{j=1}^n x_j \frac{\partial \bar F(\Lambda, x)}{\partial x_j} \tilde \xi_j = 0$. Thus the preceding process may be repeated replacing $\{ \xi_i \}$ with $\{ \tilde \xi_i \}$ in (\ref{E: 5}). It terminates in a finite number of steps since we have $k$ at most $\lfloor rD/(1-l_\sigma(\mu)) \rfloor$ so that it will follow that
\[
\xi_j = \sum_{l=1}^n x_l \frac{\partial \bar F}{\partial x_l} \tilde \eta_{jl}
\]
with $\{ \tilde \eta_{jl} \}$ skew-symmetric which will establish the claim that $\{ x_j \frac{\partial \bar F}{\partial x_j} \}$ is a regular sequence in $T$.

To establish (\ref{E: 8}), substitute (\ref{E: 7}) into (\ref{E: 5}). It follows then that 
\[
ord_\Lambda \left( \sum_{j=1}^n x_j \frac{\partial \bar F}{\partial x_j} \xi_j^{(i_0)} \right) \geq 0.
\]
For $\bar \lambda \in \bb F_q^*$, $pr_{\bar \lambda}: T \rightarrow \bar R^{(\bar \lambda)}$ is a surjective map of $\bb F_q$-algebras. The Laurent polynomials $\{ x_j \frac{\partial \bar F(\Lambda, x)}{\partial x_j} \}$ map to $\{ x_j \frac{\partial \bar F(\bar \lambda, x)}{\partial x_j} \}$; these latter form a regular sequence in $\bar R^{(\bar \lambda)}$ by Theorem \ref{T: nondeg} and \cite{AdolpSperb-ExponentialSumsand-1989}. Since the kernel of $pr_{\bar \lambda}$ is $\f m T$, it follows that there is a skew-symmetric set $\{ \gamma_{jl} \} \subset T$, homogeneous of weight $r-1-(i_0/D)(1 - l_\sigma(\mu))$, with
\[
\xi_j^{(i_0)}(x, \Lambda) = \sum_{l=1}^n x_l \frac{\partial \bar F(\Lambda, x)}{\partial x_l} \gamma_{jl}(\Lambda, x) + \Lambda^{1/D} \zeta_j(\Lambda, x)
\]
where $\zeta_j(\Lambda, x)$ is homogeneous of total weight $r - \left(\frac{i_0 + 1}{D} \right)(1 - l_\sigma(\mu))$. This gives the desired relation (\ref{E: 8}) with $\tilde \xi_j = \Lambda^{(i_0 +1)/D} \zeta_j$ and $\eta_{jl} = \Lambda^{i_0/D} \gamma_{jl}$. 

This shows $H^i(\Omega^\bullet(T, \nabla(\bar F))) = 0$ for $i \not= n$. A standard argument \cite{H-S} or \cite{AdolpSperb-ExponentialSumsand-1989} then shows that $H^i(\Omega^\bullet(T, \nabla(D))) = 0$ for $i \not= n$. To establish
\begin{equation}\label{E: 9}
T^{(r)} = V^{(r)} + \sum_{i=1}^n x_i \frac{\partial \bar F(\Lambda, x)}{\partial x_i} T^{(r-1)}
\end{equation}
it is enough to show the left side is contained in the right. In a manner altogether similar to the above argument, using downward induction on $ord_\Lambda$, we may assume that this holds for elements $\xi \in T^{(r)}$ with $ord_\Lambda(\xi) \geq (k_0 + 1)/D$. Suppose now that $\xi \in T^{(r)} $ and $ord_{\Lambda}(\xi) = k_0/D$. We may write
\[
\xi = \Lambda^{k_0/D} \xi^{(k_0)}  + \eta
\]
with $\eta \in T^{(r)}$ and $ord_\Lambda(\eta) \geq (k_0+1) / D$ and $\xi^{(k_0)}$ homogeneous of total weight $r - (k_0/D)(1 - l_\sigma(\mu))$ and $ord_\Lambda(\xi^{(k_0)}) = 0$.

Then by \cite{AdolpSperb-ExponentialSumsand-1989}
\[
pr_{\bar \lambda}(\xi^{(k_0)}) = \sum a_v x^v + \sum_{i=1}^n x_i \frac{\partial \bar F(\bar \lambda, x)}{\partial x_i} \zeta_i(x)
\]
with $v$ running over $B_{\rho}^{(\bar \lambda)}$, and $\rho = r - (k_0/D)(1 - l_\sigma(\mu))$. Let $l_{\bar \lambda}$ be the $\bb F_q$-linear map $l_{\bar \lambda}: R^{(\bar \lambda)} \rightarrow T$ which sends $x^v \mapsto (\Lambda^{m(v)} / \bar \lambda^{m(v)}) x^v$ so that $pr_{\bar \lambda} \circ l_{\bar \lambda} = 1$. Note that $l_{\bar \lambda}$ is $\bb F_q$-linear but not in general an algebra map. Then
\[
pr_{\bar \lambda}\left( \xi^{(k_0)} - \sum a_v \bar \lambda^{-m(v)} \Lambda^{m(v)} x^v - \sum_{i=1}^n x_i \frac{\partial \bar F(\Lambda, x)}{\partial x_i} l_{\bar \lambda}(\zeta_i(x)) \right)  = 0
\]
so that 
\[
\xi^{(k_0)} = \sum_{v \in B_{\bar \lambda}^{(\rho)}} a_v \bar \lambda^{-m(v)} \Lambda^{m(v)} x^v + \sum_{i=1}^n x_i \frac{\partial \bar F}{\partial x_i} l_{\bar \lambda}(\zeta_i(x)) + \Lambda^{1/D} \gamma
\]
where $\gamma \in T$, homogeneous of weight $r - (k_0+1)(1 - l_\sigma(\mu)) / D$ and $ord_{\Lambda}(\gamma) \geq (k_0 + 1)/D$. Thus
\[
\xi = (\eta + \Lambda^{(k_0 + 1)/D} \gamma) + \sum_{v \in B_{\bar \lambda}^{(\rho)}} \frac{a_v}{\bar \lambda^{m(v)}} \Lambda^{m(v) + k_0/D} x^v + \sum_{i=1}^n x_i \frac{\partial \bar F(\Lambda, x)}{\partial x_i} (\Lambda^{k_0/D} l_{\bar \lambda}(\zeta_i(x)).
\]
The result (\ref{E: 9}) then follows by the induction hypothesis applied to $\xi' = \eta + \Lambda^{(k_0 + 1)/D} \gamma$.

To establish that the sum on the right side of (\ref{E: 9}) is direct assume 
\[
v(\Lambda, x) = \sum_{i=1}^n x_i \frac{\partial \bar F(\Lambda, x)}{\partial x_i} \zeta_i(\Lambda, x).
\]
By a well-known argument, we may assume $k_0/D = ord_\Lambda(v) = \min\{ord_\Lambda \zeta_i(\Lambda, x)\}$. Let $v = \Lambda^{k_0/D} \tilde v$ and $\zeta_i = \Lambda^{k_0/D} \tilde \zeta_i$. Then
\[
\tilde v(\Lambda, x) = \sum_{i=1}^n x_i \frac{\partial \bar F}{\partial x_i} \tilde \zeta_i(\Lambda, x)
\]
so that
\[
\tilde v(\bar \lambda, x) = \sum_{i=1}^r x_i \frac{\partial \bar F(\bar \lambda, x)}{\partial x_i} \tilde \zeta_i(\bar \lambda, x)
\]
with both sides nontrivial, violating the directness of sum established in Theorem \ref{T: R decomp} for $R^{(\bar \lambda)}$. 
\end{proof}

\begin{theorem}\label{T: 3.3}
Let $\bar \lambda \in \overline{\bb F}_q^*$. Let $\tilde B := \{ x^v \mid \Lambda^{m(v)} x^v \in \bar B \}$. Let $\tilde V_{\bar \lambda}$ be the $\bb F_q^{(\bar \lambda)}$-span of $\tilde B$ in $R^{(\bar \lambda)}$ then 
\[
R^{(\bar \lambda)} = \tilde V_{\bar \lambda} \oplus \sum_{i=1}^n x_i \frac{\partial \bar F(\bar \lambda, x)}{\partial x_i} R^{(\bar \lambda)}.
\]
In particular, the monomial basis $\tilde B$ ($= B^{(\bar \lambda_0)}$ for our initial fixed chosen $\bar \lambda = \bar \lambda_0 \in \bb F_q^*$) is independent of the choice of $\bar \lambda$.
\end{theorem}

\begin{proof}
$\tilde V_{\bar \lambda} = pr_{\bar \lambda}(V \otimes_{\bb F_q} \bb F_q^{(\bar \lambda)})$. It follows that $R^{(\bar \lambda)}$ is equal to the sum on the right. It remains to show that the sum on the right is direct. Suppose
\[
v(x) = \sum_{i=1}^n x_i \frac{\partial \bar F(\bar \lambda, x)}{\partial x_i} \zeta_i(x).
\]
Using $l_{\bar \lambda}$ we lift to $T \otimes_{\bb F_q} \bb F_q^{(\bar \lambda)}$ obtaining
\[
l_{\bar \lambda}(v) - \sum_{i=1}^n x_i \frac{\partial \bar F(\Lambda, x)}{\partial x_i} l_{\bar \lambda}(\zeta_i(x)) = \Lambda^{1/D} \eta(\Lambda, x)
\]
with $l_{\bar \lambda}(v) \in V \otimes \bb F_q^{(\bar \lambda)}$. We may decompose $\eta(\Lambda, x) = \hat v + \sum_{i=1}^n x_i \frac{\partial \bar F(\Lambda, x)}{\partial x_i} \gamma_i$ according to the sum in (\ref{E: 9}). We then obtain by directness
\[
l_{\bar \lambda}(v) - \Lambda^{1/D} \hat v(\Lambda, x) = 0.
\]
But  applying $pr_{\bar \lambda}$ to this equality it follows that $v = 0$. This implies the sum is direct.
\end{proof}

\section{The case $l_\sigma(\mu) > 1$}\label{S: 3}

In this case we prefer  to treat the family in the form
\[
\bar F(\Lambda, x) = \bar f(x) + \Lambda^{-1}x^\mu.
\]
We then replace hypothesis $H(iv)$  by $H(iv)'$:

\begin{itemize}
\item[] H(iv$)'$. $\mu$ is an interior point of $Cone(\bar f)$ and $l_\sigma(\mu) > 1$.
\end{itemize}

It is not strictly necessary to assume $\mu$ is an interior point of $Cone(\bar{f})$, but without this assumption we no longer have an isobaric deformation (in particular, the set of vertices of $\Delta_{\infty}(\bar{f})$ will not be a subset of the vertices of $\Delta_{\infty}(\bar{f}, \mu))$. For simplicity of  exposition then we assume $\mu$ is an interior point of $Cone(\bar f)$. Under the hypothesis $H(iv)'$ every codimension 2 face $\tau$ of $\Delta_{\infty}(\bar{f})$ is visible from $\mu$ (so $\Gamma = \Gamma_1$ here) and $\Delta(\bar{f})$ is not a codimension one face at $\infty$ of $\Delta_\infty(\bar f, \mu)$. Then the codimension one faces of $\Delta_\infty(\bar f, \mu)$ at $\infty$ are all of the form $C(\tau, \mu)$ for the codimension 2 faces $\tau$ at $\infty$ of $\Delta_\infty(\bar f)$. We write (as in Section \ref{S: 1}) $\phi^{(\tau)}(x) \in \bb Z[x_1, \ldots, x_n]$ for the integral linear form (with relatively prime coefficients) spanned by $\tau \cup \{0\}$, with the signs adjusted so that  $Cone(\bar f) = \{ x \in \bb R^n \mid \phi^{(\tau)}(x) \geq 0 \text{ for all $\tau$ codimension 2 face at $\infty$ of $\Delta_\infty(\bar f)$} \}$. Under hypothesis $H(iv)'$ we have $\phi^{(\tau)}(\mu) > 0$ for all $\tau \in \Gamma$. We note for emphasis that hypothesis $H(v)$ now reads $p \nmid \prod_{\tau \in \Gamma} \phi^{(\tau)}(\mu)$. If $Cone(\tau, \mu)$ is the cone over $C(\tau, \mu)$ in $\bb R^n$, then the weight function in this subcone of $\Delta_\infty(\bar f, \mu)$ is given by
\[
w_{\tau, \mu}(v) = l_\sigma(v) - \frac{\phi^{(\tau)}(v)}{\phi^{(\tau)}(\mu)}(l_\sigma(\mu)-1).
\]
Similar to Theorem \ref{T: nondeg}, it is straightforward to prove the following.

\begin{theorem}
For any $\bar \lambda \in \overline{\bb F}_q^*$, the Laurent polynomial $\bar F(\bar \lambda, x) \in \bb F_q^{(\bar \lambda)}[x_1^\pm, \ldots, x_n^\pm]$ is nondegenerate with respect to $\Delta_\infty(\bar f, \mu)$. 
\end{theorem}

We also need to consider $\Delta_\infty(\bar F(\Lambda, x))$ in $\bb R^{n+1}$. In the case under consideration $\bar F(\Lambda, x)$ is quasi-homogeneous with all monomials in $Supp(\bar F)$ satisfying $W(\Lambda; x) = 1$, where
\[
W(\Lambda; x) := l_\sigma(x) + \Lambda( l_\sigma(\mu) - 1).
\]
The hyperplane $W(\Lambda, x) = 1$ meets the positive $\Lambda$-coordinate axis at the point $P_0 = ((l_\sigma(\mu) - 1)^{-1}, 0, \ldots, 0)$.  Set $\Delta_\infty(\bar F, P_0) := $ convex closure of $\Delta_\infty(\bar F) \cup \{P_0\}$ and $Cone(\bar F, P_0) := $ cone in $\bb R^{n+1}$ over $\Delta_\infty(\bar F, P_0)$. Define the extended monoid $\tilde M(\bar F) = Cone(\bar F, P_0) \cap \tilde L$, where $\tilde L$ is the extended lattice $((1/D) \bb Z) \times \bb Z^n$, where $D$ is the least common multiple of the set of positive integers $\{ \phi^{(\tau)}(\mu) \}_{\tau \in \Gamma}$. Let $e$ be the least positive integer, divisible by $D$,  such that $W(\tilde M(\bar F)) \subset(1/e)\bb Z_{\geq 0}$. 

Form $T = \bb F_q[\tilde{M}(\bar{F})]$ which is filtered in the usual way using $W$. Just as in Section \ref{S: 2}, $T$ is a free $\bb F_q[\Lambda^{1/D}]$ algebra. We describe now a basis. Just as in Section \ref{S: 1}, we assume here  each codimension 2 face $\tau$ at infinity of $\Delta_{\infty}(\bar{f})$ spans a hyperplane with equation say $\phi^{(\tau)}(x) = 0$ in $\bb R^n$, and that $Cone(\bar{F})$ in $\bb R^n$  is described by the simultaneous inequalities  $\{ \phi^{(\tau)}(x) \geq 0, \text{ for all } \tau \in \Gamma \}$. The hyperplane in $\bb R^{n+1}$ spanned by this linear space together with $\Lambda^{-1} x^{\mu}$ has equation
\[ 
\Lambda  \phi^{\tau}(\mu) + \phi^{\tau}(x) = 0
\]
and is spanned by a boundary face  of $Cone(\bar F, P_0)$. Hence $Cone(\bar{F}, P_0)$ is defined by the simultaneous inequalities 
\[
\{  \Lambda \geq  -\phi^{(\tau)}(x)/\phi^{(\tau)}(\mu), \text{ for all } \tau \in \Gamma \text{ and }  \phi^{(\tau)}(x) \geq 0, \text{ for all } \tau \in \Gamma \}.
\] 
It follows as in Section \ref{S: 2} that $T$ is a free $\bb F_q$-algebra with basis $\Lambda^r x^v$ with $(r;v) \in \tilde{M}(\bar{F})$. Furthermore it is a free $(S = \bb F_q[\Lambda^{1/D}])$-algebra with basis  $\Lambda^{m(v)} x^v$, where we have set $m(v) := - \phi^{(\tau)}(v) / \phi^{(\tau)}(\mu)$ for $v \in Cone(\tau, \mu) \cap \bb Z^n \subset Cone(\bar f) \cap \bb Z^n$ in $\bb R^n$. (Note $m(v) \leq 0$ in the present case. Note as well that $(r,v) \in \tilde{M}(\bar{F})$ is characterized by $v \in Cone(\bar F, \mu), r \geq m(v)$.) We may factor any monomial $\Lambda^r x^v$ in $\tilde M(\bar F)$ as $\Lambda^r x^v = \Lambda^{r-m(v)}( \Lambda^{m(v)} x^v)$ and define $W_\Lambda$ on $\{ \Lambda^{s/D} \mid s \in \bb Z_{\geq 0} \}$ via $W_\Lambda(s/D) = (s/D)(l_\sigma(\mu) - 1)$. Then the total weight may be written as
\begin{align*}
W(r, v) &= l_\sigma(v) - r(l_\sigma(\mu) - 1) \\
&=   w(v) + W_\Lambda(r - m(v)) .
\end{align*}
Finally, we may use here the $\bb F_q$-linear map (taking $\bar \lambda \in \bb F_q$ for convenience) $pr_{\bar \lambda}: T \rightarrow \bar R^{(\bar \lambda)}$ via $\Lambda^r x^v \mapsto (\Lambda^{r-m(v)}  |_{\Lambda = 0}) \bar \lambda^{m(v)} x^v$ which is surjective with kernel $\f mT = \Lambda^{1/D} T$ and induces on the quotient $\bar pr_{\bar \lambda}: \bar T = T / \f m T \rightarrow \bar R^{(\bar \lambda)}$ an isomorphism of graded $\bb F_q$-algebras. With this map, constructing complexes $\Omega^\bullet(T, \nabla(\bar F))$ and $\Omega^\bullet(T, \nabla(D))$ of $S$-algebras as in Section \ref{S: 2}, we conclude that Theorems \ref{T: 3.1}, \ref{T: 3.2}, and \ref{T: 3.3} hold here as well using similar arguments.

\section{$p$-adic theory}\label{S: 4}

Fix $\zeta_p$ a primitive $p$-th root of unity. Let $\bb Q_q$ be the unramified extension of $\bb Q_p$ of degree $a = [\bb F_q: \bb F_p]$ and let $\bb Z_q$ be its ring of integers. Then $\bb Z_q[\zeta_p]$ and $\bb Z_p[\zeta_p]$ are the ring of integers of $\bb Q_q(\zeta_p)$ and $\bb Q_p(\zeta_p)$, respectively. It is convenient to utilize totally ramified extensions, say $K$ and ${K_0}$, of $ \bb Q_q(\zeta_p)$ and $\bb Q_p(\zeta_p)$ respectively; let $\bb Z_q[\tilde \pi] $ and $\bb Z_p[\tilde \pi]$ be the respective ring of integers of these extension fields. We  accomplish this by adjoining an appropriate root of $\pi$ (where $\pi$ is a zero of $\sum_{j=0}^\infty t^{p^j} / p^j$ satisfying $ord_p(\pi) = 1/(p-1)$), say $\tilde \pi$, so that  any fractional valuations that may arise are in the value group of these fields, . Recall $W_\Lambda(r/D ) = (r/D)(1 - l_\sigma(\mu))$.  Set
\[
\c O_0 := \left\{ \sum_{r=0}^\infty C(r) \Lambda^{r/D} \pi^{W_\Lambda(r/D)} \mid C(r) \in \bb Z_q[\tilde \pi], C(r) \rightarrow 0 \text{ as } r \rightarrow \infty \right\}.
\]
We endow $\c O_0$ with a valuation via
\[
\left| \sum_{r=0}^\infty C(r) \Lambda^{r/D} \pi^{W_\Lambda(r/D)} \right| := \sup_{r \geq 0} |C(r)|.
\]
Under the reduction map mod $\tilde \pi$, $\c O_0$ maps onto $S= \bb F_q[\Lambda^{1/D}]$ by sending $\sum_{r=0}^\infty C(r) \Lambda^{r/D} \pi^{W_\Lambda(r/D)} \mapsto \sum \bar C(r) \Lambda^{r/D}$, and the map
\[
\c O_0 /\tilde \pi \c O_0 \rightarrow S
\]
is an isomorphism of $\bb F_q$-algebras. Let 
\[
\c C_0 := \left\{ \sum_{v \in M(\bar f, \mu)} \xi(v) \pi^{w(v)} \Lambda^{m(v)} x^v \mid \xi(v) \in \c O_0, \xi(v) \rightarrow 0 \text{ as } w(v) \rightarrow \infty \right\},
\]
a $p$-adic Banach $\c O_0$-algebra. Then the reduction map mod $\tilde \pi$ taking
\[
\sum_{v \in M(\bar f, \mu)} \xi(v) \pi^{w(v)} \Lambda^{m(v)} x^v \mapsto \sum_{v \in M(\bar f, \mu)} \bar \xi(v) \Lambda^{m(v)} x^v
\]
identifies the $S$-algebras
\[
\c C_0 /\tilde  \pi\c C_0 \rightarrow T.
\]
We now construct a complex of $p$-adic spaces similar to the characteristic $p$ complexes constructed in Section \ref{S: 1}. Set $\gamma_0 :=1$ and for $i \geq 1$, 
\begin{align*}
\gamma_i &:= \pi^{-1} \sum_{j=0}^i \pi^{p^j} / p^j \\
&= - \pi^{-1} \sum_{j=i+1}^\infty \pi^{p^j} / p^j.
\end{align*}
From this second description, we see that
\[
ord_p(\gamma_i) = \frac{p^{i+1}-1}{p-1} - (i+1)
\]
for every $i \geq 0$. Let $F(\Lambda, x) \in \bb Z_q[\Lambda, x_1^\pm, \ldots, x_n^\pm]$ be the lifting of $\bar F(\Lambda, x)$ using Teichm\"uller units for all coefficients. Define
\[
F^{(\infty)}(\Lambda, x) = \sum_{i=0}^\infty \gamma_i F^{\sigma^i}(\Lambda^{p^i}, x^{p^i})
\]
where $\sigma \in Gal(K / K_0)$ is the Frobenius automorphism of $ Gal(\bb Q_q/ \bb Q_p)$ extended to $K$ by setting $\sigma (\tilde \pi) = \tilde \pi$. Note that $F^{\sigma^i}(\Lambda^{p^i}, x^{p^i})$ has total weight $W$ less than or equal to $p^i$ for all $i \geq 0$. Since
\[
x_l \frac{\partial (F^{\sigma^i}(\Lambda^{p^i}, x^{p^i}))}{\partial x_l} = p^l x_l \frac{\partial(F^{\sigma^i})}{\partial x_l}(\Lambda^{p^i}, x^{p^i})
\]
and
\[
ord_p(\pi \gamma_i p^i) - p^i/(p-1) = p^i - 1 \geq 0,
\]
it follows that multiplication by $\pi x_l \frac{\partial F^{(\infty)}(\Lambda, x)}{\partial x_l}$ defines an endomorphism of $\c C_0$. We then define a complex of $\c O_0$-modules $\Omega^\bullet(\c C_0, \nabla(D^{(\infty)}))$ by
\[
\Omega^i(\c C_0, \nabla(D^{(\infty)})) := \bigoplus_{1\leq j_1 < j_2 < \cdots < j_i \leq n} \c C_0 \frac{dx_{j_1}}{x_{j_1}} \wedge \cdots \wedge \frac{dx_{j_i}}{x_{j_i}}
\]
with boundary map
\[
\nabla(D^{(\infty)})(\xi \frac{dx_{j_1}}{x_{j_1}} \wedge \cdots \wedge \frac{dx_{j_i}}{x_{j_i}}) = \left( \sum_{l=1}^n D_l^{(\infty)}(\xi) \frac{dx_l}{x_l} \right) \wedge \frac{dx_{j_1}}{x_{j_1}} \wedge \cdots \wedge \frac{dx_{j_i}}{x_{j_i}},
\]
where
\[
D_l^{(\infty)} := x_l \frac{\partial}{\partial x_l} + \pi x_l \frac{\partial F^{(\infty)}}{\partial x_l}.
\]

Since we will sometimes want to consider the analogous situation with $\Lambda$ replaced by a power $\Lambda^q$ it is useful sometimes to incorporate $\Lambda$ as a subscript in our notation. In such cases we write 
\[
  D_{l,\Lambda}^{(\infty)}:= D_l^{(\infty)} = \frac{1}{\exp \pi F^{(\infty)}(\Lambda, x)} \circ x_l \frac{\partial}{\partial x_l} \circ \exp \pi F^{(\infty)}(\Lambda, x).
\] 
The reduction mod $\tilde \pi$ of this complex is precisely the complex $\Omega^\bullet(T, \nabla( D))$ of $S$-algebras of Section \ref{S: 2}. Using the same argument as in \cite[Theorem 3.3]{H-S}, or the one that appears in \cite[Theorem A1]{AdolpSperb-ExponentialSumsand-1989}, and goes back at least to Monsky \cite{Monsky-p-adicanalysisand-1970}, we have the following.

\begin{theorem}
The complex $\Omega^\bullet(\c C_0, \nabla( D^{(\infty)}))$ is acyclic except in top dimension $n$ and $H^n(\Omega^\bullet(\c C_0, \nabla(D^{(\infty)})))$ is a free $\c O_0$-module of rank equal to $n! Vol \> \Delta_\infty(\bar f, \mu)$. Furthermore, 
\[
\c C_0 = \sum_{(m(v); v) \in B} \c O_0 \pi^{m(v)} \Lambda^{m(v)} x^v \oplus \sum_{l=1}^n D_l^{(\infty)}(\Lambda, x)) \c C_0
\]
where
\[
D_l^{(\infty)} = x_l \frac{\partial}{\partial x_l} + \pi x_l \frac{\partial F^{(\infty)}(\Lambda, x)}{\partial x_l}.
\]
\end{theorem}

\section{Frobenius}\label{S: 5}

We define
\begin{align}\label{E: Frobs}
\alpha_1 &:= \sigma^{-1} \circ \frac{1}{\exp \pi F^{(\infty)}(\Lambda^p, x)} \circ \psi_p \circ \exp \pi F^{(\infty)}(\Lambda, x) \\
\alpha &:= \frac{1}{\exp \pi F^{(\infty)}(\Lambda^q, x)} \circ \psi_q \circ \exp \pi F^{(\infty)}(\Lambda, x) \notag
\end{align}
where $\psi_p$ and $\psi_q$ are defined by
\begin{align*}
\psi_p( \sum A(v) x^v) &= \sum A(pv) x^v \\
\psi_q( \sum A(v) x^v ) &= \sum A(qv) x^v,
\end{align*}
and $\sigma \in Gal(\bb Q_q(\zeta_p) / \bb Q_p(\zeta_p))$ is the Frobenius automorphism. Since formally
\[
D_{l,\Lambda}^{(\infty)}= \frac{1}{\exp \pi F^{(\infty)}(\Lambda, x)} \circ x_l \frac{\partial}{\partial x_l} \circ \exp \pi F^{(\infty)}(\Lambda, x)
\]
the following commutation laws will hold for $l = 1, 2, \ldots, n$:
\begin{equation}\label{E: comm}
q D_{l, \Lambda^q}^{(\infty)} \circ \alpha = \alpha \circ D_{l, \Lambda}^{(\infty)} \quad \text{and} \quad p D_{l, \Lambda^p}^{(\infty)} \circ \alpha_1 = \alpha_1 \circ D_{l, \Lambda}^{(\infty)}.
\end{equation}
 Since the differential operators commute with $\alpha$ (respectively $\alpha_1$) up to the change from $\Lambda$ to $\Lambda^q$ (respectively $\Lambda$ to $\Lambda^p$), we introduce new spaces based on $\Lambda^q$ (respectively $\Lambda^p$) instead of $\Lambda$. If we view $\Lambda^r = (\Lambda^q)^{r/q}$ for $r \in (1/D)\bb Z_{\geq 0}$ then it makes sense to define
\[
W_{\Lambda^q}(r) := (r/q)(1 - l_\sigma(\mu)) = W_\Lambda(r/q) = (1/q)W_\Lambda(r),
\]
where on the right $W_{\Lambda} $ is defined on the further extended lattice  $((1/qD) \bb Z) \times \bb Z^n$ in the modified cone using the defining formula (\ref{E: 4}). So define for $q = p^a$ ($a \in \bb Z_{\geq 0}$, including $a = 0$)
\[
\c O_{0, q} := \left\{ \sum_{r=0}^\infty A(r) \Lambda^{r/D} \pi^{W_{\Lambda^q}(r/D)} \mid A(r) \in \bb Z_q[\tilde \pi], A(r) \rightarrow 0 \text{ as } r \rightarrow \infty \right\}.
\]
 Note that $W_{\Lambda^q}(r) \leq W_\Lambda(r)$ so $\c O_{0,q}$ is an $\c O_0$-algebra. We  define a discrete valuation on $\c O_{0, q}$ by
\[
\left|   \sum_{r=0}^\infty A(r) \Lambda^{r/D} \pi^{W_{\Lambda^q}(r/D)}  \right| := \sup_{r \geq 0} \{ |A(r)| \}.
\]
We also define
\[
\c C_0( \c O_{0, q}) := \left\{ \sum_{v \in M(\bar f, \mu)} \xi(v) \pi^{w(v)} \Lambda^{m(v)} x^v \mid \xi(v) \in \c O_{0,q}, \xi(v) \rightarrow 0 \text{ as } w(v) \rightarrow \infty \right\}
\]
with valuation
\[
\left|     \sum_{v \in M(\bar f, \mu)} \xi(v) \pi^{w(v)} \Lambda^{m(v)} x^v  \right| := \sup_{v \in M(\bar f, \mu)} \{ |\xi(v)| \}.
\]
The reduction map
\[
\sum A(r) \pi^{W_{\Lambda^q}(r/D)} \Lambda^{r/D} \mapsto \sum \bar A(r) \Lambda^{r/D}
\]
gives an isomorphism 
\[
\c O_{0, q} / \pi^{1/(qD)} \c O_{0,q} \rightarrow S_q
\]
where $S_q = \bb F_q[\Lambda^{1/D}]$ but with grading defined by $W_{\Lambda^q}$. 

We replace the total weight $W$ by $W_q$ where $W_q(r; v) := W_{\Lambda^q}(r - m(v)) + w(v)$ taking values in $(1/(qe)) \bb Z_{\geq 0}$. We have the following analogue of our previous results.

\begin{theorem}\label{T: q-version}
Let $D^{(\infty)}_{l, \Lambda^q} = x_l \frac{\partial}{\partial x_l} + \pi x_l \frac{\partial F^{(\infty)}(\Lambda^q, x)}{\partial x_l}$. Let $\Omega^\bullet(\c C_0(\c O_{0,q}), \nabla(D_{\Lambda^q}^{(\infty)}))$ be the complex
\[
\Omega^i := \bigoplus_{1 \leq j_1 < \cdots < j_i \leq n} \c C_0( \c O_{0,q}) \frac{d x_{j_1}}{x_{j_1}} \wedge \cdots \wedge \frac{d x_{j_i}}{x_{j_i}}
\]
with boundary map
\[
\nabla(D_{\Lambda^q}^{(\infty)})(\xi \frac{d x_{j_1}}{x_{j_1}} \wedge \cdots \wedge \frac{d x_{j_i}}{x_{j_i}}) = \left( \sum_{l=1}^n D_{l, \Lambda^q}(\xi) \frac{d x_l}{x_l} \right) \wedge \frac{d x_{j_1}}{x_{j_1}} \wedge \cdots \wedge \frac{d x_{j_i}}{x_{j_i}}.
\]
This complex is acyclic except in top dimension $n$ and $H^n(\Omega^\bullet(\c C_0(\c O_{0, q}), \nabla(D_{\Lambda^q}^{(\infty)})))$ is a free $\c O_{0, q}$-module of rank equal to $n! Vol \> \Delta_\infty(\bar f, \mu)$. Furthermore,
\[
\c C_0(\c O_{0, q}) = \sum_{(m(v); v) \in \bar B} \c O_{0,q} \pi^{w(v)} \Lambda^{m(v)} x^v \oplus \sum_{l=1}^n D_{l, \Lambda^q}^{(\infty)} \c C_0(\c O_{0, q})
\]
where $\bar B$ is the same monomial basis as in Theorem \ref{T: 3.2}.
\end{theorem}

To obtain precise information needed concerning the $p$-adic sizes of the entries of the matrix of Frobenius, we proceed with Dwork's approach in the present context. As in \cite{H-S}, for $0 < b \leq p/(p-1)$ and $c \in \bb R$ we define
\begin{align*}
R(b, c) &:= \left\{ \sum_{r = 0}^\infty A(r) \Lambda^{r/D} \mid A(r) \in \bb Q_q(\tilde \pi), ord_p \> A(r) \geq b W_\Lambda(r/D) + c \right\} \\
R(b) &:= \bigcup_{c \in \bb R} R(b, c).
\end{align*}
We define a valuation on $R(b)$ as follows. If $\xi = \sum_{r=0}^\infty A(r) \Lambda^{r/D} \in R(b)$, then
\begin{align*}
ord_p \> \xi &:= \inf_{r \geq 0} \{ ord_p \> A(r) - W_\Lambda(r/D) b \} \\
&= \sup \{ c \in \bb R \mid \xi \in R(b, c) \}.
\end{align*}
Note that
\[
R(b, c) R(b, c') \subset R(b, c+c')
\]
as in \cite{H-S}. Let $\c R$ be any ring with $p$-adic valuation. Set
\[
L(b, c; \c R) := \left\{ \sum_{v \in M(\bar f, \mu)} \xi(v) \Lambda^{m(v)} x^v \mid \xi(v) \in \c R, ord_p \> \xi(v) \geq b w(v) + c \right\}.
\]
In particular, if $\c R = R(b', c')$ (with $c' \geq 0$ so that $\c R$ is a ring), we may write $\xi(v) = \sum_{r \in (1/D) \bb Z_{\geq 0}} A(r, v) \Lambda^r$, and $ord_p \> A(r, v) \geq b' W_\Lambda(r) + c'$ so that in this case
\[
L(b, c; \c R) := \left\{ \sum_{r \in (1/D)\bb Z_{\geq 0}, v \in M(\bar f, \mu)} A(r,v) \Lambda^{r + m(v)} x^v \mid ord_p \> A(r,v) \geq W_\Lambda(r) b' + w(v) b + c \right\}.
\]
This motivates our definition of the spaces $K(b', b; c)$ below. Let $0 < b, b' \leq p/(p-1)$ be rational numbers and let $c \in \bb R$. Define the spaces
\begin{align*}
K(b', b; c) &:= \left\{ \sum_{(r,v) \in \tilde M(\bar F)} A(r, v) \Lambda^r x^v \mid A(r,v) \in \bb Q_q(\tilde \pi), ord_p \> A(r,v) \geq b' W_\Lambda(r - m(v)) + b w(v) + c \right\} \\
K(b', b) &:= \bigcup_{c \in \bb R} K(b', b; c).
\end{align*}

We consider Frobenius maps $\alpha$ and $\alpha_1$ on $\c C_0$ and on $K(b',b)$. Recall,  the Artin-Hasse series $E(t) := \exp\left( \sum_{j=0}^\infty t^{p^j} / p^j \right)$. Let $\pi$ be a zero of $\sum_{j=0}^\infty t^{p^j}/p^j$ satisfying $ord_p(\pi) = 1/(p-1)$. Dwork's infinite splitting function $\theta(t)$ is defined by
\[
\theta(t) := E(\pi t) = \sum_{j=0}^\infty \theta_j t^j,
\]
where it is well-known that the coefficients satisfy $ord_p \theta_j \geq j/(p-1)$. Write
\[
\bar F(\Lambda, x) = \sum_{v \in Supp(\bar f)} \bar A(v) x^v + \Lambda x^\mu \in \bb F_q[\Lambda, x_1^\pm, \ldots, x_n^\pm]
\]
(of course if we are in the case of Section \ref{S: 3} we have $\Lambda^{-1} x^\mu$ instead of $\Lambda x^\mu$). Let
\[
F(\Lambda, x) := \sum_{v \in Supp(\bar f)} A(v) x^v + \Lambda x^\mu \in \bb Z_q[\Lambda, x_1^\pm, \ldots, x_n^\pm]
\]
where  $A(v)$ is the Teichm\"uller lift of $\bar A(v)$ for each $v \in Supp(\bar f)$. It follows from the estimates on the coefficients of $\theta_j$ that
\[
\mathfrak{F}(\Lambda, x) := \theta(\Lambda x^\mu) \cdot \prod_{v \in Supp(\bar f)} \theta(A(v) x^v) \in K(1/(p-1), 1/(p-1); 0) \subset K(b'/p, b/p; 0).
\]
If we set 
\[
\mathfrak{F}_a(\Lambda, x) := \prod_{i=0}^{a-1} \mathfrak{F}^{\sigma^i}(\Lambda^{p^i}, x^{p^i})
\]
then
\[
\mathfrak{F}_a(\Lambda, x) \in K(\frac{p}{q(p-1)}, \frac{p}{q(p-1)}; 0) \subset K(b' / q, b / q; 0).
\]
Since $\mathfrak{F}$ has been defined so that
\[
\mathfrak{F}(\Lambda,x) = \exp(\pi (F^{(\infty)}(\Lambda, x) - F^{(\infty)}(\Lambda^p, x^p)))
\]
and 
\[
 \mathfrak{F}_{a}(\Lambda,x) = \exp(\pi (F^{(\infty)}(\Lambda, x) - F^{(\infty)}(\Lambda^q, x^q)))
\]
then
\[
 \alpha_1 = \sigma^{-1} \circ \psi_p \circ \mathfrak{F}(\Lambda,x), \text{and }  \space \\ \alpha = \psi_q  \circ \mathfrak{F}_a(\Lambda,x).
\]
Note that
\[
\c C_0 \subset K(1/(p-1), 1/(p-1); 0) \subset K(b' / p, b / p)
\]
so that multiplication by $\mathfrak{F}$ takes $\c C_0$, as well as $K(b'/p, b/p)$, into $K(b' / p, b/p)$, and multiplication by $\mathfrak{F}_a$ takes these two spaces into $K(b'/q, b/q)$. It is easy to see 
\[
\psi_p( K(b',b ; c)) \subset K(b', pb ; c)  \qquad \text{and} \qquad \psi_q( K(b', b; c)) \subset K(b', qb; c).
\]
Finally, we note that for $b' $ and $b \geq 1/(p-1)$, $K(b'/q, b; 0) \subset \c C_0(\c O_{0,q})$. Therefore, since $\psi_p$ acts on the $x$-variables, then  $\alpha_1 = \sigma^{-1} \circ \psi_p \circ \mathfrak{F}(\Lambda, x)$ maps $\sigma^{-1}$-semilinearly $\c C_0(\c O_0)$  into $\c C_0(\c O_{0,p})$, and it maps $\sigma^{-1}$-semilinearly $K(b', b; c)$  into $K(b'/p, b; c)$. Similarly, since $\alpha = \psi_q \circ \mathfrak{F}_a(\Lambda, x)$ then $\alpha$ maps $\c C_0(\c O_0)$ into $\c C_0(\c O_{0, q})$ linearly over $\bb Z_q[\tilde \pi]$, as well as $K(b', b; c)$ into $K(b'/q, b; c)$.
We may use $\alpha_1$ and $\alpha$ to define chain maps as follows. Let
\begin{align}\label{E: Frob}
\text{Frob}_{\Lambda}^i &:= \bigoplus_{1 \leq j_1 < \cdots < j_i \leq n} q^{n-i}  \alpha \frac{dx_{j_1}}{x_{j_1}} \wedge \cdots \wedge \frac{dx_{j_i}}{x_{j_i}}, \\
\text{Frob}_{1, \Lambda}^i &:= \bigoplus_{1 \leq j_1 < \cdots < j_i \leq n} p^{n-i}  \alpha_1 \frac{dx_{j_1}}{x_{j_1}} \wedge \cdots \wedge \frac{dx_{j_i}}{x_{j_i}}. \notag
\end{align}
Then the commutation rules (\ref{E: comm}) ensure that this defines chain maps
\[
\begin{CD}
\Omega^\bullet( \c C_0( \c O_0), \nabla(D^{(\infty)}_\Lambda)) @>{\text{Frob}_{1,\Lambda}^\bullet}>> \Omega^\bullet(\c C_0(\c O_{0, p}), \nabla(D^{(\infty)}_{\Lambda^p}))
\end{CD}
\]
and
\[
\begin{CD}
\Omega^\bullet( \c C_0( \c O_0), \nabla(D^{(\infty)}_\Lambda)) @>{\text{Frob}_{\Lambda}^\bullet}>> \Omega^\bullet(\c C_0(\c O_{0, q}), \nabla(D^{(\infty)}_{\Lambda^q})).
\end{CD}
\]
Note that since $\psi_p$, respectively $\psi_q$, acts only on the $x$ variables, ${Frob}_{1,\Lambda}^\bullet$, respectively ${Frob}_{\Lambda}^\bullet$ is a chain map of $\sigma^{-1}$-semilinear maps over $\c O_0$, respectively linear maps over $\c O_0$.
All the complexes above are acyclic except possibly in dimension $n$. Note that for all $b$ such that $0 < b \leq p / (p-1)$, $\pi x_i \frac{\partial F^{(\infty)}}{\partial x_i}(\Lambda^q, x)$ belongs to $K(b/q, b; - \tilde e)$ where $\tilde e := b - \frac{1}{p-1}$. We are now in a position to prove a sequence of results as in \cite{H-S} leading to explicit estimates for the matrix of the relative Frobenius map with respect to the basis $ \bar B$. The proofs here are only slight modifications of the proofs of the results in \cite{H-S}. We define first the following space 
\[
W(b/q, b; c) := \left\{ \sum_{r \in (1/D)\bb Z_{\geq 0}, (m(v); v) \in \bar B} A(r; v) \Lambda^r ( \Lambda^{m(v)} x^v) \Bigg|
\begin{split}
&A(r; v) \in \bb Q_q(\tilde \pi) \text{ such that} \\
&ord_p A(r; v) \geq b \left( W_{\Lambda^q}(r) + w(v) \right) + c
\end{split}
\right\}.
\]
In particular,
\[
W(b/q, b; 0) = K(b/q, b; 0) \cap \bigoplus_{v \in B} \c O_{0, q} \pi^{w(v)} \Lambda^{m(v)} x^v.
\]

\begin{theorem}
(cf. Theorem 3.6 of \cite{H-S}) Let $0 < b \leq p / (p-1)$ and $c \in \bb R$. Then the following equality holds for any $q$ (a power of $p$)
\[
K(b/q, b; c) = W(b/q, b; c) + \sum_{i=1}^n \pi x_i \frac{\partial F}{\partial x_i}(\Lambda^q, x) K(b/q, b; c + \tilde e).
\]
\end{theorem}

It is useful to use the notation 
\[
W_{\Lambda^q}(r + m(v), v) := W_{\Lambda^q}(r) + w(v)
\]
for the $\Lambda^q$-total weight of the monomial $\Lambda^{r + m(v)} x^v$. Recall $F^{(\infty)}(\Lambda, x) = \sum_{i = 0}^\infty \gamma_i F^{\sigma^i}(\Lambda^{p^i}, x^{p^i})$ where $ord_p(\gamma_i) = \frac{p^{i+1}-1}{p-1} - (i+1)$ for $i \geq 0$. Recall also 
\[
D_{l, \Lambda^q}^{(\infty)} := x_l \frac{\partial}{\partial x_l} + \pi x_l \frac{\partial F^{(\infty)}}{\partial x_l}(\Lambda^q, x).
\]

\begin{theorem}
For $b$ a rational number satisfying $1/(p-1) < b < p/(p-1)$, $c \in \bb R$, and $q$ a power of $p$, we have:
\begin{enumerate}
\item (cf. Theorem 3.7 of \cite{H-S}) 
\[
K(b/q, b; c) = W(b/q, b; c) + \sum_{l=1}^n \pi x_l \frac{\partial F^{(\infty)}}{\partial x_l}(\Lambda^q, x) K(b/q, b; c+\tilde e).
\]
\item (cf. Theorem 3.8 of \cite{H-S}) 
\[
K(b/q, b; c) = W(b/q, b; c) + \sum_{l=1}^n D_{l, \Lambda^q}^{(\infty)} K(b/q, b; c+ \tilde e).
\]
\item  (cf. Theorem 3.9 of \cite{H-S})
\[
K(b/q, b) = W(b/q, b) \oplus \sum_{l=1}^n D_{l, \Lambda^q}^{(\infty)} K(b/q, b).
\]
\end{enumerate}
\end{theorem}

We are now able to give the following estimates for the relative Frobenius map. 

\begin{theorem}\label{T: FrobEst}
Let $b$ be a rational number, $1/(p-1) < b \leq p/(p-1)$. If we write for each  $ u \in B$, i.e., $\Lambda^{m(u)} x^u \in \bar B$,
\[
\alpha_1( \Lambda^{m(u)} x^u) = \sum_{v \in B} A(u, v) \Lambda^{m(v)} x^v \quad \text{mod } \sum_{l=1}^n D_{l, \Lambda^p}^{(\infty)} K(b/p, b),
\]
then
\[
A(u, v) \in R(b/p; (b/p)(p w(v) - w(u))).
\]
\end{theorem}
\begin{proof} For each $u \in B$ the monomial $\Lambda^{m(u)} x^u$ belongs to $K(b'/p,b/p, -w(u)(b/p))$ where $b'$ and $b$ are rational numbers $1/(p-1)<b,b' \leq p/((p-1)$. $\mathfrak F \in K(1/(p-1), 1/(p-1), 0)$ implies that $\mathfrak F \Lambda^{m(u)} x^u \in K(b'/p,b/p, -w(u)(b/p))$. It follows that $\alpha_1(\Lambda^{m(u)} x^u) \in K(b'/p,b,-w(u)(b/p))$. Writing 
\[
\alpha_1(\Lambda^u x^u) = \sum_{v \in B} \left(\sum_{r \in (1/D){\bb Z}_{\geq 0}} A_r(u,v) \Lambda^r \right) \Lambda^{m(v)} x^v + \eta
\]
with 
\[
\sum_{v \in B} \left(\sum_{r \in (1/D){\bb Z}_{\geq 0}} A_r(u,v) \Lambda^r \right) \Lambda^{m(v)} x^v \in W(b'/p,b,-w(u)(b/p)) 
\]
and
\[
   \eta \in  \sum_{l=1}^n D_{l, \Lambda^q}^{(\infty)} K(b'/p, b; -w(u)(b/p)+ \tilde e)
\]
so that $ord_p(A_r(u,v) \geq b'W_{\Lambda^p}(r) + bw(v) -(b/p)w(u)$. For $b = b'$, $ord_p(A_r(u,v) \geq (b/p)(W_{\Lambda}(r) + pw(v) -w(u)).$  These estimates imply $ \sum (A_r(u,v) \Lambda^r)$ belongs to $R(b/p; (b/p)(p w(v) - w(u))).$
\end{proof}
Taking $b= b' = p/(p-1)$ the above estimates give us the following estimates for the matrix of the Frobenius map $H^n(Frob_{1,\Lambda})$ from $H^n(\Omega^\bullet( \c C_0( \c O_0), \nabla(D^{(\infty)}_\Lambda)))$ to $H^n(\Omega^\bullet( \c C_0( \c O_{0,p}), \nabla(D^{(\infty)}_{\Lambda^p})))$ with respect to the bases $ \bar B_{\Lambda}=\{ \Lambda^{m(u)} x^u \}_{u \in B} $ of $H^n(\Omega^\bullet( \c C_0( \c O_0), \nabla(D^{(\infty)}_\Lambda)))$ and $\bar B_{{\Lambda}^p}=\{ \Lambda^{pm(u)} x^u \}_{u \in B} $ of $H^n(\Omega^\bullet( \c C_0( \c O_{0,p}), \nabla(D^{(\infty)}_{\Lambda^p})))$.

\begin{corollary}\label{C: Frob est1}
 Let $\c A(\Lambda) = (\c A(u,v)(\Lambda))$ be the matrix  (with entries in $\c O_{0,p}$) of  $H^n(Frob_{1,\Lambda})$ with respect to the bases $\bar B_{\Lambda}$ and $\bar B_{{\Lambda}^p} $. Then $\c A(u,v)(\Lambda) \in R(1/(p-1); 1/(p-1)(pw(v)-w(u))$. In particular, $ord_p(\c A(u,v)(\Lambda)) \geq (pw(v)-w(u))/(p-1)$.
\end{corollary}

\section{Lower-order deformations}\label{S: 6}

With an eye directed at some future applications we combine here the isobaric deformations described above with the lower-order deformation technique of our earlier work \cite{H-S}. In order to have a greater applicability it is useful to modify the work above as follows. We consider the following variant of the deformation above. Let

\[
\bar G(\Lambda, x) = \bar F(\Lambda^M, x) = \bar f(x) + \Lambda^M x^\mu.
\]
(We will consider here the case with $l_\sigma(\mu) < 1$; the case with $l_\sigma(\mu) > 1$ is entirely analogous.) We continue to work under the hypotheses H(i) - H(v) given in Section \ref{S: 1}. As in Section \ref{S: 1}, set $\Delta_\infty(\bar f, \mu) = $ the convex closure of $\{ 0 \} \cup Supp(\bar f) \cup \{ \mu \}$. Denote by $Cone(\bar f, \mu)$ the cone over $\Delta_\infty(\bar f, \mu)$, and $M(\bar f, \mu) = Cone(\bar f, \mu) \cap \bb Z^n$. For $v \in M(\bar f, \mu)$, denote by $w(v)$ the polyhedral weight defined on $M(\bar f, \mu)$ by the polyhedron $\Delta_\infty(\bar f, \mu)$. 

The changes in the preceding work resulting from the change replacing $\Lambda$ with $\Lambda^M$ are not major and occur mostly in the work on the total space in $\bb R^{n+1}$ and the consequent effects on the $p$-adic theory. We consider $\Delta_\infty(\bar G) = $ convex closure of $\{ 0 \} \cup Supp(\bar G)$ in $\bb R^{n+1}$. Let $Cone(\bar G)$ be the cone in $\bb R^{n+1}$ over $\Delta_\infty(\bar G)$. In $\bb R^{n+1}$, $\bar G$ is quasi-homogeneous and nondegenerate with respect to $\Delta_\infty(\bar G)$ just as in Section \ref{S: 2} above. In fact all elements $(r, v)$ of $Supp(\bar G)$ satisfy $W_{\bar G}(r, v) = 1$ where
\[
W_{\bar G}(\Lambda, x) := l_\sigma(x) + \Lambda(1 - l_\sigma(\mu)) / M.
\]
We assume that the rational coefficients of $W_{\bar G}(\Lambda, x)$ all lie in $(1 / \tilde d) \bb Z$, $\tilde d > 0$. 

The hyperplane $W_{\bar G}(\Lambda, x) = 1$ meets the positive $\Lambda$-coordinate axis at the point $P_0 = ( \frac{M}{1 - l_\sigma(\mu)}, 0, \ldots, 0)$.  Set $\Delta_\infty(\bar G, P_0) := $ convex closure of $\Delta_\infty(\bar G) \cup \{P_0\}$, and $Cone(\bar G, P_0) := $ cone in $\bb R^{n+1}$ over $\Delta_\infty(\bar G, P_0)$. The boundary faces of $Cone(\bar G, P_0)$ are given by $\Lambda = 0$ and 
\[
\Lambda M^{-1} \phi^{(\tau)}(\mu) - \phi^{(\tau)}(x) = 0
\]
where $\phi^{(\tau)}(x)$ is the integral form such that the affine subspace of $\bb R^n$ spanned by $\tau$, a codimension one face of $\Delta(\bar f)$ is given by the equation $\phi^{(\tau)}(x) = 0 $ . Then the modified cone $Cone(\bar G, P_0)$ is described by the simultaneous inequalities
\[
\begin{cases}
\Lambda \geq 0 \\
\Lambda \geq M \phi^{(\tau)}(x) / \phi^{(\tau)}(\mu) & \text{for } \tau \in \Gamma_1,\\
\phi^{(\tau)}(x) \geq 0 & \text{for } \tau \in \Gamma - \Gamma_1
\end{cases}
\]
where we recall that $\Gamma$ is the collection of all codimension one faces of $\Delta(\bar f)$, and $\Gamma_1$ are those visible from $\mu$.

For each $\tau \in \Gamma_1$, let $\tilde D(\tau, M) := \frac{\phi^{(\tau)}(\mu)}{gcd(\phi^{(\tau)}(\mu), M)}$. Let $\tilde D$ be the least common multiple of $\{ \tilde D(\tau, M)\}_{\tau \in \Gamma_1}$. For any $u \in M(\bar f, \mu)$ we define $m_M(u) := M \phi^{(\tau)}(u) / \phi^{(\tau)}(\mu) \in (1 / \tilde D)\bb Z_{\geq 0}$. Then the extended monoid of $\bar G$  is $\tilde M(\bar G) = (\tilde D^{-1} \bb Z \times \bb Z^n) \cap Cone(\bar G, P_0)$. Setting $\tilde e := \tilde d \tilde D$, then $W_{\bar G}(\tilde M(\bar G)) \subset \tilde e^{-1} \bb Z_{\geq 0}$.

With these changes in place, we now proceed to lower order deformations of the type studied in \cite{H-S}. Let $\bar P$ be the Laurent polynomial
\[
\bar P(T, \Lambda, x) = \sum \bar p(\gamma; r; v) T^\gamma \Lambda^r x^v \in \bb F_q[T_1^\pm, \ldots, T_s^\pm, \Lambda, x_1^\pm, \ldots, x_n^\pm].
\]
We consider the family
\[
\bar H(T, \Lambda, x) := \bar G(\Lambda, x) + \bar P(T, \Lambda, x)
\]
For each monomial $\kappa = T^\gamma \Lambda^r x^u  = (\gamma; r; u) \in Supp(\bar P)$ we say the $x$-support of $\kappa$ is $u$, and $(r; u)$ is the $\Lambda$-$x$- support of $\kappa$. We assume for each such monomial $\kappa$ in $Supp(\bar P)$, that the $x$-support $u$ is in $Cone(\bar f, \mu)$ and the $\Lambda$-$x$- support $(r; u)$ of $\kappa$ is in $Cone(\bar G, P_0)$ and satisfies $W_{\bar G}(r; u) < 1$. The latter requirement is why we call this a lower order deformation.

In this case we construct a relative polytope $\Upsilon$ in $\bb R^s$ as in \cite{H-S}. $\Upsilon$ is defined as the convex hull of the points
\[
\{0\} \cup \{ \left(\frac{1}{1 - W_{\bar G}(r; u)} \right) \gamma \mid (\gamma; r; u) \in Supp(\bar P) \}.
\]
Let $Cone(\Upsilon)$ be the cone in $\bb R^s$ over $\Upsilon$.

In constructing a good ring of coefficients for relative cohomology, we construct a modified relative polytope in $\bb R^{s+1}$. We order the coordinates of $\bb R^{s+1}$ by $(T, \Lambda)$.  Let $\tilde \Upsilon$ be the convex closure in $\bb R^{s+1}$ of
\[
\{ 0\} \cup \{ \left(\frac{1}{1 - W_{\bar G}(r; u)} \right) \gamma \mid (\gamma; r; u) \in Supp(\bar P) \} \cup \{ (0, \ldots, 0; M / (1 - l_\sigma(\mu))) \}.
\]
Note that the point of intersection of the affine hyperplane $W_{\bar G}(\Lambda, x) = 1$ with the $\Lambda$-axis in $\bb R^{n+1}$ occurs at $\Lambda = M/(1-l_{\sigma}(\mu)$. 

Let $Cone(\tilde \Upsilon)$ be the cone in $\bb R^{s+1}$ over $\tilde \Upsilon$, and $\tilde M(\tilde \Upsilon) =  \bb Z^s \times (\tilde D^{-1} \bb Z)$ the extended monoid on which we define a weight function
\[
W_{\tilde \Upsilon}(\gamma; r) := w_\Upsilon(\gamma) + \frac{r}{M}(1 - l_\sigma(\mu))
\]
where $w_\Upsilon $ is the usual polyhedral weight determined by the polytope $\Upsilon $ in $\bb R^s$, and $\frac{r}{M}(1 - l_\sigma(\mu))$ is the weight on $\bb F_q[\Lambda^{1 / \tilde D}]$ determined by the line segment $[0, M/(1 - l_\sigma(\mu))]$. Let $w_{\tilde \Upsilon}$ be the polyhedral weight on $\tilde M(\tilde \Upsilon)$ determined by $\tilde \Upsilon $. We identify $M(\Upsilon) \times \tilde M(\bar G)$ and $\tilde M(\tilde \Upsilon) \times M(\bar f, \mu)$ via
\[
(\gamma; (r; u)) \longmapsto ((\gamma, r - m_M(u)), u).
\]
Let 
\begin{align*}
W(\gamma; r; u) &:=  W_\Upsilon(\gamma) + W_{\bar G}(r, u) \\
&= W_{\tilde \Upsilon}(\gamma; r - m_M(u)) + w(u).
\end{align*}
Using this weight function, $\c T := \bb F_q[M(\Upsilon) \times \tilde M(\bar G)]$ is a graded $\bb F_q$-algebra. Let $\c S = \bb F_q[\tilde M(\tilde \Upsilon)]$ be graded using $w_{\tilde \Upsilon}$. Then $\c T$ is a free graded $\c S$-algebra with free basis $\bar B (= \bar B_{\Lambda}):= \{ \Lambda^{m_M(u)} x^u \mid u \in M(\bar f, \mu) \}$.

As above in Section \ref{S: 2} consider the complexes $\Omega^\bullet(\c T, \nabla(\bar H))$ and $\Omega^\bullet(\c T, \nabla(D_{\bar H}))$ where the spaces for the complexes are the same: 
\[
\Omega^i := \bigoplus_{1 \leq j_1 < j_2 < \cdots < j_i \leq n} \c T \frac{dx_{j_1}}{x_{j_1}} \wedge \cdots \wedge \frac{dx_{j_i}}{x_{j_i}}
\]
for $i = 0, \ldots, n$, but the boundary maps are defined, respectively, by
\[
\nabla(\bar H)( \xi \frac{dx_{j_1}}{x_{j_1}} \wedge \cdots \wedge \frac{dx_{j_i}}{x_{j_i}}) = \left( \sum_{l =1}^n x_l \frac{\partial \bar H}{\partial x_l} \xi \frac{d x_l}{x_l} \right) \wedge \frac{dx_{j_1}}{x_{j_1}} \wedge \cdots \wedge \frac{dx_{j_i}}{x_{j_i}}
\]
and
\[
\nabla(D_{\bar H})( \xi \frac{dx_{j_1}}{x_{j_1}} \wedge \cdots \wedge \frac{dx_{j_i}}{x_{j_i}}) = \left( \sum_{l =1}^n D_{\bar H, l}(\xi) \frac{d x_l}{x_l} \right) \wedge \frac{dx_{j_1}}{x_{j_1}} \wedge \cdots \wedge \frac{dx_{j_i}}{x_{j_i}},
\]
where
\[
D_{\bar H, l} = x_l \frac{dx_l}{x_l} + x_l \frac{\partial \bar H}{\partial x_l}(T, \Lambda, x).
\]
The next result follows from the arguments in Section \ref{S: 2} and \cite{H-S}. 

\begin{theorem}
The complexes $\Omega^\bullet(\c T, \nabla(\bar H))$ and $\Omega^\bullet(\c T, \nabla(D_{\bar H}))$ are acyclic except in top dimension $n$. $H^n(\Omega^\bullet(\c T, \nabla(\bar H)))$ is a free graded $\c S$-algebra of finite rank $n! Vol(\Delta_\infty(\bar f, \mu))$. $H^n(\Omega^\bullet(\c T, \nabla(D_{\bar H})))$ is a free filtered $\c S$-algebra of finite rank $n! Vol(\Delta_\infty(\bar f, \mu))$. Let $\tilde B$ be the basis from Theorem \ref{T: 3.3}. Let $\bar B := \{ \Lambda^{m_M(u)} x^u \mid u \in \tilde B \}$ and $\c V$ the free $\c S$-module with basis $\bar B$. Then
\[
\c T = \c V \oplus \sum_{l = 1}^n x_l \frac{\partial \bar H}{\partial x_l}(T, \Lambda, x) \c T
\]
and
\[
\c T = \c V \oplus \sum_{l=1}^n D_{\bar H, l} \c T.
\]
\end{theorem}

We may copy as well the $p$-adic theory of Section \ref{S: 5} above but with some slight modifications. Define the coefficient ring
\[
\tilde{\c O}_0 := \left\{ \xi = \sum_{\gamma \in M(\Upsilon), r \in \tilde D^{-1} \bb Z_{\geq 0}} C(\gamma; r) T^\gamma \Lambda^r  \pi^{W_{\tilde \Upsilon}(\gamma; r)} \mid C(\gamma; r) \in \bb Z_q[\tilde \pi], C(\gamma, r) \rightarrow 0 \text{ as } w_{\tilde \Upsilon}(\gamma; r) \rightarrow \infty \right\}.
\]
equipped with the sup-norm on the coefficients $|\xi| := \sup_{(\gamma; r)} |C(\gamma; r)|$. Define
\[
\tilde{\c C}_0 = \left\{ \sum_{u \in M(\bar f, \mu)} \xi(u) \pi^{w(u)} \Lambda^{m_M(u)} x^u \mid \xi(u) \in \tilde{\c O}_0, \xi(u) \rightarrow 0 \text{ as } w(u) \rightarrow \infty \right\}.
\]
Let $H(T, \Lambda, x)$ be the lifting of $\bar H(T, \Lambda, x)$ by lifting each coefficient using the Teichm\"uller units. Set
\[
H^{(\infty)}(T, \Lambda, x) := \sum_{i = 0}^\infty \gamma_i H^{\sigma^i}(T^{p^i}, \Lambda^{p^i}, x^{p^i})
\]
with $\gamma_i$ and $\sigma$ as in Section \ref{S: 5}. We have then the following $p$-adic complex $\Omega^\bullet(\tilde{\c C}_0, \nabla(\tilde D^{(\infty)}))$ defined by
\[
\Omega^i := \bigoplus_{1 \leq j_1 < j_2 < \cdots < j_i \leq n} \tilde{\c C}_0 \frac{dx_{j_1}}{x_{j_1}} \wedge \cdots \wedge \frac{dx_{j_i}}{x_{j_i}}
\]
with boundary operator
\[
\nabla(\tilde D^{(\infty)})( \eta \frac{dx_{j_1}}{x_{j_1}} \wedge \cdots \wedge \frac{dx_{j_i}}{x_{j_i}}) = \left( \sum_{l =1}^n \tilde D_l^{(\infty)}(\eta) \frac{d x_l}{x_l} \right) \wedge \frac{dx_{j_1}}{x_{j_1}} \wedge \cdots \wedge \frac{dx_{j_i}}{x_{j_i}},
\]
where
\[
\tilde D^{(\infty)}_l = x_l \frac{dx_l}{x_l} + \pi x_l \frac{\partial H^{(\infty)}}{\partial x_l}.
\]

\begin{theorem}
The complex $\Omega^\bullet(\tilde{\c C}_0, \nabla(\tilde D^{(\infty)}))$ is acyclic except in top dimension $n$, and $H^n(\Omega^\bullet(\tilde{\c C}_0, \nabla(\tilde D^{(\infty)})))$ is a free $\tilde{\c O}_0$-module of rank equal to $n! Vol(\Delta_\infty(\bar f, \mu))$. Furthermore,
\[
\tilde{\c C}_0 = \sum_{(m_M(u), u) \in \bar B} \tilde{\c O}_0 \pi^{w(u)} \Lambda^{m_M(u)} x^u \oplus \sum_{l=1}^n \tilde D_l^{(\infty)} \tilde{\c C}_0.
\]
\end{theorem}

Next, we define the Frobenius, following closely the development of Section \ref{S: 5}. Define
\begin{align*}
\alpha_1 &:= \sigma^{-1} \circ \frac{1}{\exp \pi H^{(\infty)}(T^p, \Lambda^p, x)} \circ \psi_p \circ \exp \pi H^{(\infty)}(T, \Lambda, x) \\
\alpha &:= \frac{1}{\exp \pi H^{(\infty)}(T^q, \Lambda^q, x)} \circ \psi_q \circ \exp \pi H^{(\infty)}(T, \Lambda, x),
\end{align*}
where $\psi_p$ and $\psi_q$ are defined as in Section \ref{S: 5}. Since formally
\[
\tilde D_{l,T, \Lambda}^{(\infty)}:= \frac{1}{\exp \pi H^{(\infty)}(T, \Lambda, x)} \circ x_l \frac{\partial}{\partial x_l} \circ \exp \pi H^{(\infty)}(T, \Lambda, x)
\]
the following commutation laws will hold for $l = 1, 2, \ldots, n$:
\begin{equation}\label{E: comm2}
q \tilde D_{l, T^q, \Lambda^q}^{(\infty)} \circ \alpha = \alpha \circ \tilde D_{l, T, \Lambda}^{(\infty)} \quad \text{and} \quad p \tilde D_{l, T^p \Lambda^p}^{(\infty)} \circ \alpha_1 = \alpha_1 \circ \tilde D_{l, T, \Lambda}^{(\infty)}.
\end{equation}
For $q = p^a$ ($a \in \bb Z_{\geq 0}$), define the weight function
\[
W_{\tilde \Upsilon, q}(\gamma; r) := \frac{1}{q} \left( w_{\tilde \Upsilon}(\gamma) + \frac{r}{M}(1 - l_\sigma(\mu)) \right),
\]
and the coefficient space
\[
\tilde{\c O}_{0, q} := \left\{ \xi = \sum_{\gamma \in M(\Upsilon), r \in \tilde D^{-1} \bb Z_{\geq 0}} C(\gamma; r) T^\gamma \Lambda^r  \pi^{W_{\tilde \Upsilon, q}(\gamma; r)} \mid C(\gamma; r) \in \bb Z_q[\tilde \pi], C(\gamma, r) \rightarrow 0 \text{ as } (\gamma; r) \rightarrow \infty \right\}
\]
equipped with the sup-norm on the coefficients $|\xi| := \sup_{(\gamma; r)} |C(\gamma; r)|$. Define
\[
\tilde{\c C}_0(\tilde{\c O}_{0,q}) = \left\{ \sum_{u \in M(\bar f, \mu)} \xi(u) \pi^{w(u)} \Lambda^{m_M(u)} x^u \mid \xi(u) \in \tilde{\c O}_{0, q}, \xi(u) \rightarrow 0 \text{ as } w(u) \rightarrow \infty \right\}.
\]
Analogous to Theorem \ref{T: q-version}, we have:

\begin{theorem}
Let $\tilde D^{(\infty)}_{l, T^q, \Lambda^q} = x_l \frac{\partial}{\partial x_l} + \pi x_l \frac{\partial H^{(\infty)}(T^q, \Lambda^q, x)}{\partial x_l}$. Let $\Omega^\bullet(\tilde{\c C}_0(\tilde{\c O}_{0,q}), \nabla(\tilde D_{T^q, \Lambda^q}^{(\infty)}))$ be the complex
\[
\Omega^i := \bigoplus_{1 \leq j_1 < \cdots < j_i \leq n} \tilde{\c C}_0( \tilde{\c O}_{0,q}) \frac{d x_{j_1}}{x_{j_1}} \wedge \cdots \wedge \frac{d x_{j_i}}{x_{j_i}}
\]
with boundary map
\[
\nabla(\tilde D_{T^q, \Lambda^q}^{(\infty)})(\xi \frac{d x_{j_1}}{x_{j_1}} \wedge \cdots \wedge \frac{d x_{j_i}}{x_{j_i}}) = \left( \sum_{l=1}^n \tilde D_{l, T^q, \Lambda^q}(\xi) \frac{d x_l}{x_l} \right) \wedge \frac{d x_{j_1}}{x_{j_1}} \wedge \cdots \wedge \frac{d x_{j_i}}{x_{j_i}}.
\]
This complex is acyclic except in top dimension $n$ and $H^n(\Omega^\bullet(\tilde{\c C}_0(\tilde{\c O}_{0, q}), \nabla(\tilde D_{T^q, \Lambda^q}^{(\infty)})))$ is a free $\tilde{\c O}_{0, q}$-module of rank equal to $n! Vol \> \Delta_\infty(\bar f, \mu)$. Furthermore,
\[
\tilde{\c C}_0(\tilde{\c O}_{0, q}) = \sum_{(m(v); v) \in \bar B} \tilde{\c O}_{0,q} \pi^{w(v)} \Lambda^{m(v)} x^v \oplus \sum_{l=1}^n \tilde D_{l, T^q, \Lambda^q}^{(\infty)} \tilde{\c C}_0(\tilde{\c O}_{0, q})
\]
where $\bar B$ is the same monomial basis as in Theorem \ref{T: 3.2}.
\end{theorem}

Next, for $0 < b \leq p/(p-1)$ and $c \in \bb R$, define the spaces
\begin{align*}
\tilde R(b; c) &:= \left\{  \sum_{\gamma \in M(\Upsilon), r \in \tilde D^{-1} \bb Z_{\geq 0}} A(\gamma; r) T^\gamma \Lambda^r  \mid A(\gamma; r) \in \bb Q_q(\tilde \pi), ord_p(A(\gamma, r)) \geq b W_{\tilde \Upsilon, q}(\gamma, r) + c  \right\} \\
\tilde R(b) &:= \bigcup_{c \in \bb R} \tilde R(b; c),
\end{align*}
and for $0 < b' \leq p/(p-1)$,
\[
\tilde K(b', b; c) := \left\{ \sum_{ \substack{\gamma \in M(\Upsilon), r \in \tilde D^{-1} \bb Z_{\geq 0} \\ u \in M(\bar f, \mu)}} A(\gamma, r, u) T^\gamma \Lambda^r x^u \Bigg|  
\begin{split}
&A(\gamma, r, u) \in \bb Q_q(\tilde \pi) \text{ such that} \\ 
&ord_p(A(\gamma, r, u)) \geq b' W_{\tilde \Upsilon}(\gamma, r - m_M(u)) + b w(v) + c
\end{split}
\right\}
\]
and
\[
\tilde K(b', b) := \bigcup_{c \in \bb R} \tilde K(b', b; c).
\]
Writing
\[
H(T, \Lambda, x) = \sum A(\gamma, r, u) T^\gamma \Lambda^r x^u,
\]
we define using the splitting function $\theta$ from Section \ref{S: 5}
\[
\mathfrak{H}(T, \Lambda, x) := \prod_{(\gamma, r, u) \in Supp(H)} \theta(A(\gamma, r, u) T^\gamma \Lambda^r x^u) \in \tilde K(1/(p-1), 1/(p-1); 0) \subset \tilde K(b'/p, b/p; 0).
\]
Set
\[
\mathfrak{H}_a(T, \Lambda, x) := \prod_{i=0}^{a-1} \mathfrak{H}^{\sigma^i}(T^{p^i}, \Lambda^{p^i}, x^{p^i}),
\]
then
\[
\mathfrak{H}_a \in \tilde K(\frac{p}{q(p-1)}, \frac{p}{q(p-1)}; 0) \subset \tilde K(b'/p, b/p; 0).
\]
Note that $\tilde {\c C}_0 \subset \tilde K(1 / (p-1), 1/(p-1);0) \subset \tilde K(b'/p, b/p)$, and for $b', b \geq 1/(p-1)$, $\tilde K(b'/q, b;0) \subset \tilde {\c C}_0(\tilde{\c O}_{0, q})$. It follows that $\alpha_1 := \sigma^{-1} \circ \psi_p \circ \mathfrak{H}$ maps $\sigma^{-1}$-semilinearly $\tilde{\c C}_0$ into $\tilde{\c C}_0(\tilde {\c O}_{0, p})$, and it maps $\tilde K(b', b; c)$ $\sigma^{-1}$-semilinearly into $\tilde K(b'/p, b; c)$. Similarly, $\alpha = \psi_q \circ \mathfrak{H}_a$ maps $\tilde{\c C}_0(\tilde {\c O}_0)$ into $\tilde {\c C}_0(\tilde{\c O}_{0, q})$ linearly over $\bb Z_q[\tilde \pi]$, and $\tilde K(b', b; c)$ into $\tilde K(b'/q, b; c)$. We may now define chain maps as follows. Let
\begin{align*}
\text{Frob}^i &:= \bigoplus_{1 \leq j_1 < \cdots < j_i \leq n} q^{n-i}  \alpha \frac{dx_{j_1}}{x_{j_1}} \wedge \cdots \wedge \frac{dx_{j_i}}{x_{j_i}}, \\
\text{Frob}_1^i &:= \bigoplus_{1 \leq j_1 < \cdots < j_i \leq n} p^{n-i}  \alpha_1 \frac{dx_{j_1}}{x_{j_1}} \wedge \cdots \wedge \frac{dx_{j_i}}{x_{j_i}}. \notag
\end{align*}
Then the commutation rules (\ref{E: comm2}) ensure that this defines chain maps
\[
\begin{CD}
\Omega^\bullet( \tilde {\c C}_0( \tilde {\c O}_0), \nabla(\tilde D^{(\infty)}_{T, \Lambda})) @>{\text{Frob}_1^\bullet}>> \Omega^\bullet(\tilde{\c C}_0(\tilde{\c O}_{0, p}), \nabla(\tilde D^{(\infty)}_{T^p, \Lambda^p}))
\end{CD}
\]
and
\[
\begin{CD}
\Omega^\bullet( \tilde{\c C}_0( \tilde{\c O}_0), \nabla(\tilde D^{(\infty)}_{T, \Lambda})) @>{\text{Frob}^\bullet}>> \Omega^\bullet(\tilde{\c C}_0(\tilde {\c O}_{0, q}), \nabla(\tilde D^{(\infty)}_{T^q, \Lambda^q})).
\end{CD}
\]
As in the previous section,  these complexes  are acyclic except possibly in dimension $n$. Following that argument, we obtain
the following estimates for the relative Frobenius map. 

\begin{theorem}\label{T: FrobEst2}
If we write for each $u \in \tilde B$, i.e for $\Lambda^{m_M(u)} x^u \in B$,
\[
\alpha_1( \Lambda^{m_M(u)} x^u) = \sum_{v \in B} A(u, v)(T,\Lambda) \Lambda^{m_M(v)} x^v \quad \text{mod } \sum_{l=1}^n D_{l, T^p, \Lambda^p}^{(\infty)} \tilde K(b/p, b),
\]
then
\[
A(u, v)(T,\Lambda) \in \tilde R(b/p; (b/p)(p w(v) - w(u))).
\]
\end{theorem}
Clearly a corollary entirely analogous to Corollary \ref{C: Frob est1} holds here as well with the new definition of weight being the key  modification.

\section{$L$-functions} \label{S: L-function}

Much of the material here references the work in \cite[Section 3.3]{H-S}. Let $\bar F(\Lambda, x) = \bar f(x) + \Lambda x^\mu$ be the Laurent polynomial from Section \ref{S: 1}, or $\bar F(\Lambda, x) = \bar f(x) + \Lambda^{-1} x^\mu$ from Section \ref{S: 3}. Let $\bar \lambda \in \overline{\bb F}_q^*$, with $deg(\bar \lambda) := [\bb F_q(\bar \lambda) : \bb F_q]$. Using Dwork's splitting function, define an additive character $\Theta : \bb F_q \rightarrow \overline{\bb Q}_p$ by $\Theta := \theta(1)^{Tr_{\bb F_q / \bb F_p}(\cdot)}$, and $\Theta_{\bar \lambda} := \Theta \circ Tr_{\bb F_q(\bar \lambda) / \bb F_q}$. We consider the toric exponential sums
\[
S_r(\bar F, \bar \lambda)  := \sum_{x \in \bb F_{q^{r deg(\bar \lambda)}}^{*n}} \Theta_{\bar \lambda} \circ Tr_{\bb F_{q^{r deg(\bar \lambda)}} / \bb F_q(\bar \lambda)} \left(\bar F(\bar \lambda, x) \right)
\]
and the associated $L$-functions
\[
L(\bar F, {\bar \lambda}, T) := L(\bar F, {\bar \lambda}, \Theta, \bb G_m^n / \bb F_q({\bar \lambda}), T) := \exp \left( \sum_{r = 1}^\infty S_r({\bar F}, {\bar \lambda}) \frac{T^r}{r} \right).
\]
Let $\lambda$ be the Teichm\"uller representative of ${\bar \lambda}$. Let $\c O_{0,  \lambda} := \bb Z_q[\tilde \pi,  \lambda]$. There is a  ring map, $sp_{\lambda}$ which we call the specialization map at $ \lambda$, from $\c O_0$ to $\c O_{0,  \lambda}$ induced by the map sending $\Lambda \mapsto  \lambda$. Similarly, let $\c C_{0,  \lambda}$ be the $\c O_{0,  \lambda}$-module obtained by specializing the space $\c C_0$ at $\Lambda =  \lambda$. We recall, for $\bar \lambda \in \overline{\bb F}_q^*$ and $\lambda$ its  Teichm\"uller representative, that the complex $\Omega^\bullet(\c C_{0,\lambda}, \nabla( D^{(\infty)}_{\lambda}))$ defined just as the complex  $\Omega^\bullet(\c C_0, \nabla( D^{(\infty)}))$ in Section \ref{S: 4} but with $\c C_0 $ replaced by $\c C_{0, \lambda}$ and $D_{l, \Lambda}^{(\infty)} = x_l \frac{\partial}{\partial x_l} + \pi x_l \frac{\partial F^{(\infty)}(\Lambda, x)}{\partial x_l}$ replaced by $D_{l,\lambda}^{(\infty)} := x_l \frac{\partial}{\partial x_l} + \pi x_l \frac{\partial F^{(\infty)}(\lambda, x)}{\partial x_l}$ was studied in  \cite{AdolpSperb-ExponentialSumsand-1989} and shown there to be acyclic except in top dimension n. Furthermore, as in  \cite{AdolpSperb-ExponentialSumsand-1989}, we define $\mathfrak{F}(\lambda, x) := sp_{\lambda} \mathfrak{F}(\Lambda, x)$, $\mathfrak{F}_a(\lambda,x) := sp_{\lambda} \mathfrak{F}_a(\Lambda,x)$ and set $\alpha_{1,\lambda}  :=  \sigma^{-1} \circ \psi_p  \circ \mathfrak{F}(\lambda, x) $ and $\alpha_{\lambda} := \alpha_{1,\lambda}^{adeg(\bar \lambda)}$.We may define $Frob_{ \lambda}^{\bullet}$  acting as a chain map on   $\Omega^\bullet(\c C_{0,\lambda}, \nabla( D^{(\infty)}_{\lambda}))$ as in (\ref{E: Frob}) but with $\alpha$ replaced by $\alpha_{ \lambda}$ and $q$ replaced by $q^{deg(\lambda)}$.  Then using Dwork's trace formula, we obtain
\begin{align*}
S_r(\bar F, {\bar \lambda}) &= (q^{r deg({\bar \lambda})} - 1)^n Tr(\alpha_{ \lambda} \mid \c C_{0, \lambda}) \\
&= \sum_{i=0}^n (-1)^i Tr( H^i(Frob_{ \lambda})^r \mid H^i(\c C_{0,  \lambda},\nabla(D_{\lambda}^{(\infty)}))).
\end{align*}
Since cohomology is acyclic by \cite{AdolpSperb-ExponentialSumsand-1989} except in top dimension $n$, we have
\[
S_r(\bar F,{\bar \lambda}) = (-1)^n Tr( H^n(Frob_{ \lambda})^r \mid H^n(\c C_{0,  \lambda}, \nabla(D_{\lambda}^{(\infty)}))).
\]
In other words,  we have $L(\bar F,{\bar \lambda}, T)^{(-1)^{n+1}} = det(1 - H^n(Frob_{\lambda})T) $. For each ${\bar \lambda} \in \overline{\bb F}_q^\times$, set $\frac A({\bar \lambda}) := \{\pi_i({\bar \lambda})\}_{i=1}^N$ the collection of eigenvalues of $ H^n(Frob_{ \lambda})$. Clearly then $L(\bar F,{\bar \lambda}, T)^{(-1)^{n+1}} = (1 - \pi_1({\bar \lambda}) T) \cdots (1 - \pi_N({\bar \lambda}) T) $ where $N := n! \> Vol(\Delta_\infty(\bar f, \mu))$. In \cite{AdolpSperb-ExponentialSumsand-1989}  it is proved that for each such ${\bar \lambda}$  the Newton polygon of $L(\bar F,{\bar \lambda}, T)^{(-1)^{n+1}} $ lies over the Newton polygon (using $ord_{q^{deg(\hat \lambda)}}$) of
\begin{equation}\label{E: NPfibre}
\prod_{v \in  B} ( 1- (q^{deg({\bar \lambda})})^{w(v)} T),
\end{equation}
where $ B (=\tilde B )$ is defined in Theorem \ref{T: 3.3}.

 For each  $\bar \lambda \in \overline{\bb F}_q^*$, let $ \mathfrak A(\bar \lambda)$ denote the multi-set of eigenvalues of $ H^n(Frob_{\lambda})$. Let $\c L$ be a linear algebra operation such as a symmetric power, or exterior power, or tensor power or a combination of such. Let $\c L \mathfrak  A({\bar \lambda})$ be the corresponding  multi-set of eigenvalues of $\c L H^n(Frob_{\lambda})$. Define
\[
L(\c L, \bar F, \bb G_m / \bb F_q, T) := \prod_{{\bar \lambda} \in |\bb G_m / \bb F_q|} \ \    \prod_{\tau({\bar \lambda}) \in \c L \mathfrak A({\bar \lambda})} (1 - \tau({\bar \lambda}) T^{deg({\bar \lambda})})^{-1}.
\]
To aid the reader, we will consider a running example throughout this section; if $\c L$ is the operation of the $k$-th symmetric power tensor the $l$-th exterior power, then
\begin{align*}
\c L \mathfrak A({\bar \lambda}) &= Sym^k \mathfrak A({\bar \lambda}) \otimes \wedge^l \mathfrak A({\bar \lambda}) \\
&= \{ \pi_1({\bar \lambda})^{i_1} \cdots \pi_N({\bar \lambda})^{i_N} \pi_{j_1}({\bar \lambda}) \cdots \pi_{j_l}({\bar \lambda}) \mid i_1 + \cdots + i_N = k, 1 \leq j_1 < \cdots < j_l \leq N \},
\end{align*}
where the latter is a multi-set, and $\mathfrak A({\bar \lambda}) = \{ \pi_l({\bar \lambda}) \mid l=1, \ldots, N \}$. In general, the cardinality of $\c L \mathfrak A({\bar \lambda})$ is independent of ${\bar \lambda}$. Let $\c L N$ denote the cardinality of $\c L \mathfrak A({\bar \lambda})$ .

Let $\bar B = \{ \Lambda^{m(v_1)} x^{v_1}, \ldots, \Lambda^{m(v_N)} x^{v_N}  \}$ be the basis for $H^n := H^n(\Omega^\bullet(\c C_0, \nabla(F)))$ as in Theorem \ref{T: 3.2}. For $q$ a power of the prime $p$ (including possibly $q = p^0$) define
\[
H_{{\Lambda}^q}^i :=  H^i( \Omega^\bullet(\c C_0(\c O_{0,q}), \nabla(D_{{\Lambda}^q}^{(\infty)}))).
\]
 For emphasis and clarity, we will in the following denote $\bar B$ by $\bar B_{\Lambda}$, and we will write $\bar B_{{\Lambda}^q}$ for the collection  of monomials  obtained by replacing $\Lambda$ by ${\Lambda}^q$ in $\bar B_{\Lambda}$.  Then  $H_{{\Lambda}^q}^n $ is a free $\c O_{0,q}$-module with basis $\bar B_{{\Lambda}^q}$. Let $\c L \bar B_{\Lambda} = \{ e^{\b b^{(l)}} \}_{l=1,2, \ldots, \c L N}$. In our example, $\c L H^n = Sym^k H^n \otimes \wedge^l H^n$, elements in the basis $\c L \bar B_{\Lambda}$ take the form
\[
e^{(\b i, \b j)} = (\Lambda^{m(v_1)} x^{v_1})^{i_1} \cdots (\Lambda^{m(v_N)} x^{v_N})^{i_N} \otimes (\Lambda^{m(v_{j_1})} x^{v_{j_1}} \wedge \cdots \wedge \Lambda^{m(v_{j_l})} x^{v_{j_l}})
\]
where $\b b = (\b i, \b j)$ with $\b i = (i_1. \cdots, i_n)$  and $\b j =(j_1, \cdots, j_l) $ satisfying
\[
i_1 + \cdots + i_N = k \quad \text{and} \quad 1 \leq j_1 < \cdots < j_l \leq N.
\]
We extend the Frobenius map $H^{n}(Frob_{1,\Lambda})$ mapping from $H^{n}_{\Lambda}$ to $H^{n}_{{\Lambda}^p}$ in the  usual way to obtain the ${\sigma}^{-1} $ semilinear (over $\c O_0)$ map
\[
\c L H^n(Frob_{1,\Lambda}): \c L H^{n}_{\Lambda} \rightarrow \c L H^{n}_{{\Lambda}^p}.
\]
Similarly, we have the  linear $\c O_0$ map
\[
 \c L H^n(Frob_{\Lambda}): \c L H^{n}_{\Lambda} \rightarrow \c L H^{n}_{{\Lambda}^q}
\]

Let $\c A(\Lambda)=(\c A_{\bold i, \bold j}(\Lambda))$ be the $N \times N$ matrix of $ H^n(Frob_{1,\Lambda}) $with respect to the  bases $\bar B_{\Lambda} $ and $ \bar B_{{\Lambda}^p}$ (see Corollary \ref{C: Frob est1} above). Let $\c L \c A(\Lambda)=(\c L \c A_{\bold i, \bold j}(\Lambda))$ be the $\c L N \times \c L N$ matrix of $\c L H^n(Frob_{1,\Lambda}) $with respect to the  bases $\c L \bar B_{\Lambda} $ and $ \c L \bar B_{{\Lambda}^p}$. Similarly let $\c B(\Lambda)$ be the matrix of $\ H^n(Frob_{\Lambda})$ with respect to the bases $ \bar B_{\Lambda}$ and  $\bar B_{{\Lambda}^q}$ and  $ \c L \c B(\Lambda)$ be the matrix of $\c L H^n(Frob_{\Lambda})$ with respect to the bases $\c L \bar B_{\Lambda}$ and  $\c L \bar B_{{\Lambda}^q}$ . It is immediate that
\[
 \c L H^n(Frob_{1,\Lambda})^a = \c L H^n(Frob_{\Lambda})
\]
so that 
\[
 \c L \c A^{\sigma^{a-1}}(\Lambda^{p^{a-1}}) \cdots  \c L \c A^\sigma(\Lambda^p) \c L \c A(\Lambda)= \c L \c B(\Lambda)
\]
It is clear from our construction that if $\bar \lambda \in \bb F_q^* $, then $\c L \c A(\lambda)$ is the matrix of $\c L H^n(Frob_{1, \lambda})$ mapping  $\c L H^n(\c C_{0,  \lambda}, \nabla(D_{\lambda}^{(\infty)})))$ into $\c L H^n(\c C_{0,  {\lambda}^p}, \nabla(D_{{\lambda}^p}^{(\infty)})))$ with respect to the bases $sp_{\lambda} \c L \bar B_{\Lambda}$ and $sp_{\lambda} \c L \bar B_{{\Lambda}^p}$. Similarly for $\c L \c B(\lambda)$. Now let $\bar \lambda \in \overline{\bb F}_q^*$ and $ \lambda$ its  Teichm\"uller lift. For a positive integer $r$, let $\c L \c B^{(r)}(\Lambda)$ be the matrix of 
\[
\c L H^n(Frob_{{\Lambda}^{q^{r-1}}}) \circ \c L H^n(Frob_{{\Lambda}^{q^{r-2}}}) \circ \cdots   \circ \c L H^n(Frob_{\Lambda})                                                                                                                                                                                                                                                                                                                                                                                                                                                                                                                                                                                                                                                                                                                                                                                                                                                                                                                                                                                                                                                                             \]
with respect to the bases $\c L \bar B_{\Lambda}$ and $\c L \bar B_{{\Lambda}^{q^r}}.$ Then $\c L \c B^{(deg(\bar \lambda))}(\lambda) ( = sp_{\lambda} \c L \c B^{(deg(\bar \lambda))}(\lambda)= \c L \c  B({\lambda}^{q^{deg(\bar \lambda)-1}}) \c L \c B({\lambda}^{q^{deg(\bar \lambda)-2}}) \cdots \c L \c B(\lambda))$ is the matrix of $\c L H^n(Frob_{\lambda}) $, an endomorphism over $\c O_{0, \lambda} $ of $\c L H^n(\c C_{0,  \lambda}, \nabla(D_{\lambda}^{(\infty)}) )$) with respect to the basis $sp_{\lambda} \bar B_{\Lambda}.$

Since  $\c L \mathfrak A(\lambda)$ is the set of eigenvalues of $\c L \c B^{(deg(\bar \lambda))}(\lambda)$, we have 
\[
det(1- \c L \c B^{(deg(\bar \lambda))}(\lambda)T) =det(1 - \c L \c B( \lambda^{q^{deg(\bar \lambda)-1}}) \cdots \c L \c B(\lambda^q) \c L \c B( \lambda) T) = \prod_{\tau(\bar \lambda) \in \c L \mathfrak A(\lambda)}(1 - \tau(\bar \lambda) T).
\]
Consequently,
\begin{align*}
L(\c L, \bar F, \bb G_m / \bb F_q, T) &:= \prod_{\bar \lambda \in |\bb G_m / \bb F_q|} \ \ \prod_{\tau(\bar \lambda) \in \c L \mathfrak A(\bar \lambda)}(1 - \tau(\bar \lambda) T^{deg(\bar \lambda)} )^{-1} \\
&= \prod_{\bar \lambda \in | \bb G_m / \bb F_q|} det(1 - \c L \c B( \lambda^{q^{deg(\lambda)}}) \cdots \c L \c B( \lambda^q) \c L \c B( \lambda) T^{deg(\bar \lambda)})^{-1}.
\end{align*}

Define a weight on each basis vector in $\c L \bar B = \{ e^{\b b} \}_{\b b \in I}$ as follows. Using (\ref{E: Wt}), define
\begin{align}\label{E: weightBasis}
\bold w(\b b) :&= \sum_{i=1}^N m_i W(m(v_i), v_i) \\
&= \sum_{i=1}^N m_i w(v_i)
\end{align}
where $\Lambda^{m(v_i)} x^{v_i}$ appears $m_i$ times in the basis element $e^{\b b}$. For example, continuing our running example, if
\[
e^{\b b} = e^{(\b i, \b j)}=(\Lambda^{m(v_1)} x^{v_1})^{i_1} \cdots (\Lambda^{m(v_N)} x^{v_N})^{i_N} \otimes (\Lambda^{m(v_{j_1})} x^{v_{j_1}} \wedge \cdots \wedge \Lambda^{m(v_{j_l})} x^{v_{j_l}})
\]
then
\[
\bold w(\b b) = i_1 w(v_1) + \cdots + i_N w(v_N) + w(v_{j_1}) + \cdots + w(v_{j_l}).
\]

\begin{proposition}\label{P: Relate A and B}
There exists a matrix $\c L \c A = (\c L \c A_{\b i, \b j})_{\b i, \b j \in I}$ with entries in $R(\frac{1}{p-1})$ such that
\begin{equation}\label{E: BFrobDecomp}
\c L \c B(\Lambda) = \c L \c A^{\sigma^{a-1}}(\Lambda^{p^{a-1}}) \cdots \c L \c A^\sigma(\Lambda^p) \c L \c A(\Lambda) \quad \text{and} \quad \c L \c A_{\b i, \b j} \in R(\frac{1}{p-1} ; \frac{1}{p-1}(p \bold {w}(\b j) - \bold {w}(\b i))).
\end{equation}
\end{proposition}

\begin{proof}

We need only to justify the estimates on the matrix entries of $\c L \c A(\Lambda)$ . We extend the weight function as follows. For $\Lambda^r e^{\b i}$, $r \in (1/D) \bb Z_{\geq 0}$ and $\b i \in I$,
\[
\bold W_q(r, \b i) := W_{\Lambda^q}(r)  + \bold  w(\b i) \in \frac{1}{D(q)} \bb Z_{\geq 0}.
\]
Define the spaces, for $q$ a power of $p$ (perhaps with $q = p^0$) and $c \in \bb R$,
\begin{align*}
\c L W(b/q, b; c) &:= \left\{ \sum_{r \in (1/D) \bb Z_{\geq 0}, \b i \in I} A(r, \b i) \Lambda^r e^{\b i} \mid A(r, \b i) \in \bb Q_q(\tilde \pi), ord_p(A(r, \b i)) \geq b \bold W_q(r, \b i) + c \right\} \\
\c L W(b/q, b) &:= \bigcup_{c \in \bb R} \c L W(b/q, b ; c).
\end{align*}
Then, for any rational $b$ satisfying $1/(p-1) < b \leq p/(p-1)$, by Theorem \ref{T: FrobEst},
\[
\c L H^n(Frob_1): \c L W(b/p, b/p; 0) \rightarrow \c L W(b/p, b; 0)
\]
and
\[
\c L \c A_{\b i, \b j} \in R(\frac{b}{p}; \frac{b}{p}( p \bold w(\b j) - \bold w(\b i) )).
\]
Setting $b = p/(p-1)$ to get the best possible $p$-adic estimates, we have
\[
\c L \c A_{\b i, \b j} \in R(\frac{1}{p-1} ; \frac{1}{p-1}(p \bold w(\b j) - \bold w(\b i))).
\]
\end{proof}

Define the order $| \c L| := r$ of a linear algebra operation $\c L$ as the least positive integer $r$ such that $\c L$ is a quotient of an $r$-fold tensor product. We now state our main result. 

\begin{theorem}\label{T: Main}(c.f. Theorem 1.1 of \cite{H-S})
For each linear algebra operation $\c L$, the $L$-function $L(\c L, \bar F,  \bb G_m / \bb F_q, T)$ is a rational function:
\[
L(\c L, \bar F,  \bb G_m / \bb F_q, T) = \frac{ \prod_{i=1}^R (1 - \alpha_i T) }{ \prod_{j=1}^S (1 - \beta_j T)} \in \bb Q(\zeta_p)(T).
\]
Furthermore, writing this in reduced form ($\alpha_i \not= \beta_j$ for every $i$ and $j$):
\begin{enumerate}
\item[(a)] The reciprocal zeros and poles $\alpha_i$ and $\beta_j$ are algebraic integers. For each reciprocal pole $\beta_j$ there is a reciprocal zero $\alpha_{k_j}$ and a positive integer $m_j$ such that $\beta_j = q^{m_j} \alpha_{k_j}$. 
\item[(b)] The degree may be bounded by $0 \leq R - S \leq D / |1 - l_\sigma(\mu)|$.
\item[(c)] The total degree $R + S$ of the $L$-functions is bounded above by
\[
R+S \leq \c L N \cdot \left( \frac{D}{|1 - l_\sigma(\mu)|} \right) \cdot 5 \cdot 2^{ 1  + 2 n | \c L | }
\]
\end{enumerate}
\end{theorem}

\begin{proof}
Rationality of the $L$-function follows from an argument almost identical to \cite[Proposition 3.10]{H-S}. Next, parts (a), (b), and (c) will follow from a general result \cite[Theorem 2.2]{H-S} once we identify all the relevant data needed. Let $\Gamma$ be the polytope in $\bb R$ defined as the convex hull of the origin and the point $D / |1 - l_\sigma(\mu)|$. In the terminology of \cite{H-S}, this is the \emph{relative polytope} of our family. Next, we modify the basis $\c L B$ by the following normalization: $\tilde e^{\bf i} := \pi^{\bold w(\b i)} e^{\b i}$ for each $\b i \in I$. It follows that the matrix of $\c L H^n(Frob_{1,\Lambda}) $ with respect to this basis takes the form $\c L \tilde A(\Lambda) = (\c L \c A_{\b i, \b j}(\Lambda) \pi^{\bold w(\b i) - \bold w(\b j)} )$, with entries satisfying
\[
\c L \tilde A_{\b i, \b j}(\Lambda) := \c L \c A_{\b i, \b j}(\Lambda) \pi^{\bold w(\b i) - \bold w(\b j)} \in R(\frac{1}{p-1} ; \bold w(\b j)).
\]
Let $\c L \tilde B(\Lambda)$ be the matrix of $\c L H^n(Frob_{\Lambda})$ with respect to this normalized basis, then
\[
\c L \tilde B(\Lambda) = \c L \tilde A^{\sigma^{a-1}}(\Lambda^{p^{a-1}}) \cdots \c L \tilde A^\sigma(\Lambda^p) \c L \tilde A(\Lambda).
\]
Parts (a), (b), and (c) now follow by an argument almost identical to that of \cite[Theorem 2.2]{H-S} using the data (and notation from \cite[Section 2]{H-S}): $s = 1$,  $b = \frac{1}{p-1}$, ramification $e = p-1$, $\{ s(\b j) = (p-1) \bold w(\b j) \}_{j \in I}$. 
\end{proof}

In the case $\bar F \in \bb F_q[x_1, \ldots, x_n]$ is a polynomial, we  may similarly define
\[
L(\c L, \bar F, \bb A / \bb F_q, T) := \prod_{{\bar \lambda} \in |\bb A / \bb F_q|} \ \    \prod_{\tau({\bar \lambda}) \in \c L \mathfrak A({\bar \lambda})} (1 - \tau({\bar \lambda}) T^{deg({\bar \lambda})})^{-1}.
\]
In this case, we have the following.

\begin{theorem}\label{T: Main2}(c.f. Theorem 1.2 of \cite{H-S})
For each linear algebra operation $\c L$, the $L$-function $L(\c L, \bar F, \bb A / \bb F_q, T)$ is a rational function:
\[
L(\c L, \bar F, \bb G_m / \bb F_q, T) = \frac{ \prod_{i=1}^R (1 - \alpha_i T) }{ \prod_{j=1}^S (1 - \beta_j T)} \in \bb Q(\zeta_p)(T)
\]
which satisfies:
\begin{enumerate}
\item[(a)] The reciprocal zeros and poles $\alpha_i$ and $\beta_j$ satisfy
\[
ord_q(\alpha_i), ord_q(\beta_j) \geq \left( \frac{D}{|1 - l_\sigma(\mu)|} \right).
\]
\item[(b)] The degree may be bounded by $0 \leq R - S \leq D / |1 - l_\sigma(\mu)|$.
\item[(c)] The total degree $R + S$ of the $L$-functions is bounded above by
\[
R+S \leq \c L N \cdot \left( \frac{D}{|1 - l_\sigma(\mu)|} \right) \cdot 6 \cdot 2^{ 1  + 2 n | \c L | }.
\]
\end{enumerate}
\end{theorem}

\begin{proof}
Parts (b) and (c) follow immediately from Theorem \ref{T: Main} above and the argument on \cite[p. 557]{AdolpSperb-Newtonpolyhedraand-1987}. Part (a) follows from the general result \cite[Theorem 2.2, Part (d)]{H-S} using the data found in the proof of Theorem \ref{T: Main}.
\end{proof}

\section{$L$-functions of families with lower order deformations}\label{S: Lower deform}

Let $\bar H(T, \Lambda, x)$ be the lower deformation Laurent polynomial from Section \ref{S: 6}. Let $\bar t \in (\overline{\bb F}_q^*)^s$ and $\bar \lambda \in \overline{\bb F}_q^*$, with $deg(\bar t, \bar \lambda) :=  [\bb F_q(\bar t, \bar \lambda) : \bb F_q]$. Continuing notation from the previous section, define the additive character $\Theta : \bb F_q \rightarrow \overline{\bb Q}_p$ by $\Theta := \theta(1)^{Tr_{\bb F_q / \bb F_p}(\cdot)}$, and $\Theta_{\bar t, \bar \lambda} := \Theta \circ Tr_{\bb F_q(\bar t, \bar \lambda) / \bb F_q}$. Define
\[
S_r(\bar t, \bar \lambda)  := \sum_{x \in \bb F_{q^{r deg(\bar t, \bar \lambda)}}^{*n}} \Theta_{\bar t, \bar \lambda} \circ Tr_{\bb F_{q^{r deg(\bar t, \bar \lambda)}} / \bb F_q(\bar t, \bar \lambda)} \left(\bar H(\bar t, \bar \lambda, x) \right)
\]
and its $L$-function
\[
L(\bar H, \bar t, \bar \lambda,  T) := L(\bar H_{\bar t, \bar \lambda}, \Theta, \bb G_m^n / \bb F_q( {\bar t}, {\bar \lambda}), T) := \exp \left( \sum_{r = 1}^\infty S_r({\bar t, \bar \lambda}) \frac{T^r}{r} \right).
\]
Let $ t$ and $ \lambda$ be the Teichm\"uller representatives of $\bar t$ and ${\bar \lambda}$, respectively. For reasons essentially identical to those in Section \ref{S: L-function}, 
\[
L(\bar H, \bar t, \bar \lambda,  T)^{(-1)^{n+1}} = (1 - \pi_1({\bar t, \bar \lambda}) T) \cdots (1 - \pi_N({\bar t, \bar \lambda}) T)
\]
where $N := n! \> Vol(\Delta_\infty(\bar f, \mu))$. For each ${\bar t, \bar \lambda} \in \overline{\bb F}_q^\times$, set $\mathfrak A({\bar t, \bar \lambda}) := \{\pi_i({\bar t, \bar \lambda})\}_{i=1}^N$, the eigenvalues of $ Frob_{ t,  \lambda}$ (where $Frob_{ t,  \lambda} $ is defined in an altogether analogous  way to the definition of $Frob_{\lambda}$ given in Section \ref{S: L-function}). Let $\c L$ be a linear algebra operation. Let $\c L \mathfrak A({\bar t},{\bar \lambda})$ be the set of eigenvalues of $\c L Frob_{t,\lambda}$. Define
\[
L(\c L, \bar H, \bb {G_m}^{s+1} / \bb F_q, T) := \prod_{{\bar t, \bar \lambda} \in |\bb G_m^{s+1} / \bb F_q|} \ \    \prod_{\tau({\bar t, \bar \lambda}) \in \c L \mathfrak A({\bar t, \bar \lambda})} (1 - \tau({\bar t, \bar \lambda}) T^{deg({\bar t, \bar \lambda})})^{-1}.
\]

Define the order $| \c L| := r$ of a linear algebra operation $\c L$ as the least positive integer $r$ such that $\c L$ is a quotient of an $r$-fold tensor product. Let $\tilde s$ denote the dimension of the smallest linear subspace of $\bb R^s$ which contains $\Upsilon$, and denote by $Vol(\Upsilon)$ the volume of $\Upsilon$ in this linear subspace with respect to Haar measure normalized so that a fundamental domain of the integer lattice in the subspace has unit volume.

\begin{theorem}\label{T: lower L}
For each linear algebra operation $\c L$, the $L$-function $L(\c L, \bar H, \bb G_m^{s+1} / \bb F_q, T)$ is a rational function:
\[
L(\c L, \bar H, \bb G_m^{s+1} / \bb F_q, T)^{(-1)^s} = \frac{ \prod_{i=1}^R (1 - \alpha_i T) }{ \prod_{j=1}^S (1 - \beta_j T)} \in \bb Q(\zeta_p)(T).
\]
Furthermore, writing this in reduced form ($\alpha_i \not= \beta_j$ for every $i$ and $j$):
\begin{enumerate}
\item[(a)] The reciprocal zeros and poles $\alpha_i$ and $\beta_j$ are algebraic integers. For each reciprocal pole $\beta_j$ there is a reciprocal zero $\alpha_{k_j}$ and a positive integer $m_j$ such that $\beta_j = q^{m_j} \alpha_{k_j}$. 
\item[(b)] If $\tilde s < s$ then $R = S$, else if $\tilde s = s$ then
\[
0 \leq R - S \leq  \c LN \cdot (s+1)! \cdot Vol(\Upsilon) \cdot \left( \frac{M \tilde D}{|1 - l_\sigma(\mu)|} \right).
\]
\item[(c)] The total degree $R + S$ of the $L$-functions is bounded above by
\[
R+S \leq \c L N \cdot (\tilde s +1)! \> Vol(\Upsilon) \cdot \left( \frac{M \tilde D}{|1 - l_\sigma(\mu)|} \right) \cdot 2^{ \tilde s +1   + (1 + \frac{1}{\tilde s +1}) n | \c L | } (1 + 2^{1 + \frac{1}{\tilde s+1}})^{s+1}.
\]
\end{enumerate}
\end{theorem}

\bibliographystyle{amsplain}
\bibliography{References-1}

\end{document}